\documentclass[11pt,letterpaper,reqno]{amsart}
\usepackage{amscd,amsfonts,mathrsfs,amsthm,amssymb, amsmath,mathtools}
\usepackage{upgreek}
\usepackage{csquotes}
\usepackage{enumerate}
\usepackage{bbm}
\usepackage{multicol}
\usepackage{color}
\usepackage{array}
\usepackage{ae}
\usepackage{accents}
\usepackage{epsfig}
\usepackage[e]{esvect}
\usepackage[nobysame]{amsrefs}

\let\oldtocsection=\tocsection
\let\oldtocsubsection=\tocsubsection
\let\oldtocsubsubsection=\tocsubsubsection
\renewcommand{\tocsection}[2]{\hspace{0em}\oldtocsection{#1}{#2}}
\renewcommand{\tocsubsection}[2]{\hspace{1em}\oldtocsubsection{#1}{#2}}
\renewcommand{\tocsubsubsection}[2]{\hspace{2em}\oldtocsubsubsection{#1}{#2}}

\newcommand{\R}{\mathbb{R}}

\def\real     #1{{\mathbb R^{#1}}}

\def\natural  #1{{\mathbb N^{#1}}}

\def\dt       {{\partial}_{t}}

\def\equationcolor {\color{black}}
\def\textcolor     {\color{black}}

\def\bcoleq    {\begin{equation}\equationcolor}
\def\ecoleq    {\textcolor\end{equation}}
\def\bcoleqn   {\equationcolor\begin{eqnarray}}
\def\ecoleqn   {\end{eqnarray}\textcolor}

\def\gm{{\operatorname{g}_M}}
\def\gn{{\operatorname{g}_N}}

\def\gk{{\operatorname{g}_{M\times N}}}

\def\gp{{\operatorname{g}_{P}}}
\def\rm{{\operatorname{R}_M}}
\def\rn{{\operatorname{R}_N}}
\def\rk{{\operatorname{R}_{M\times N}}}
\def\sk{{\operatorname{s}_{M\times N}}}
\def\rind{\operatorname{R}}
\def\sind{\operatorname{S}}
\def\tind{\operatorname{T}}
\def\pind{\operatorname{p}}

\def\dF{dF}
\def\df{df}

\def\gind{\operatorname{g}}

\def\Ric{\operatorname{Ric}}

\def\BR{\operatorname{BRic}}

\def\e{\operatorname{euc}}

\DeclareMathOperator*{\Scal}{Scal}

\DeclareMathOperator*{\Id}{Id}
\DeclareMathOperator*{\trace}{trace}
\DeclareMathOperator*{\rank}{rank}

\newtheorem{theorem}{Theorem}[section]
\newtheorem{mythm}{Theorem}

\newtheorem{lemma}[theorem]{Lemma}

\newtheorem{mycor}[mythm]{Corollary}

\newtheorem{proposition}[theorem]{Proposition}
\newtheorem{definition}[theorem]{Definition}
\theoremstyle{definition}
\newtheorem{remark}[theorem]{Remark}

\numberwithin{equation}{section}

\begin{document}

\title[Graphical mean curvature flow with bounded bi-Ricci curvature]{Graphical mean curvature flow with bounded bi-Ricci curvature}
\author[Renan Assimos]{\textsc{Renan Assimos}}
\address{%
	\hspace{-14pt} R. Assimos, Leibniz University Hannover,
	Institute of Differential Geometry,
	Welfengarten 1,
	30167 Hannover, Germany}
	\email{renan.assimos@math.uni-hannover.de}

\author[Andreas Savas-Halilaj]{\textsc{Andreas Savas-Halilaj}}
\address{%
         \hspace{-12pt}A. Savas-Halilaj,
	University of Ioannina,
	Department of Mathematics,
	Section of Algebra and Geometry,
	45110 Ioannina, Greece}
\email{ansavas@uoi.gr}

\author[Knut Smoczyk]{\textsc{Knut Smoczyk}}
\address{%
	\hspace{-12pt}K. Smoczyk,\!
	Leibniz University Hannover,\!
	Institute of Differential Geometry,
	Riemann Center for Geometry and Physics,
	Welfengarten 1,
	30167 Hannover, Germany}
	\email{smoczyk@math.uni-hannover.de}

\date{}

\thanks{The second author is supported by
HFRI: Grant 133, and the third by DFG SM 78/7-1}

\begin{abstract}
We consider the graphical mean curvature flow of strictly area decreasing maps $f:M\to N$, where $M$
is a compact Riemannian manifold of dimension $m> 1$ and $N$ a complete Riemannian surface of bounded geometry. We prove long-time existence of the flow and that the strictly area decreasing property is preserved, when the bi-Ricci curvature $\BR_M$ of $M$ is bounded from below by the sectional curvature $\sigma_N$ of $N$. In addition,
	we obtain smooth convergence to a minimal map if
	$\Ric_M\ge\sup\{0,{\sup}_N\sigma_N\}$.
	These results significantly improve known results on the graphical mean curvature flow in codimension $2$. 
\end{abstract}
\maketitle

\section{Introduction and summary}
	\noindent Suppose $f:M\to N$ is a smooth map between the Riemannian manifolds $M$ and $N$ and let
	$$\varGamma_f:=\big\{(x,f(x))\in M\times N:x\in M\big\}$$
	denote the {\em graph} of $f$. We deform $\varGamma_f$ by the mean curvature flow. Some general questions are whether the flow stays graphical, it exists for all times, and it converges to a minimal graphical submanifold $\varGamma_{\infty}$ generated by a smooth map $f_\infty:M\to N$. In this case, $f_\infty$ is called a {\em minimal map} and can be regarded as a {\em canonical representative} of the homotopy class of $f$.
	
	The first result concerning the evolution of graphs by its mean curvature was obtained by Ecker and Huisken \cite{ecker}. They proved long-time existence of the mean curvature flow of entire graphical hypersurfaces in the euclidean space and convergence to flat subspaces under the assumption that the graph is {\em straight} at infinity. For maps between arbitrary Riemannian manifolds the situation is more complicated. However, under suitable conditions on the differential of
	$f$ and on the curvatures of $M$ and $N$, it is possible to establish long-time existence and convergence of the graphical mean curvature flow; for example see \cite{lubbe1,ss2,ss1,ss0,stw,tsuiwang,wang}.
	
	A smooth map $f:M\to N$ between Riemannian manifolds is called {\em strictly area decreasing}, if 
	$$|\df(v)\wedge \df(w)\vert<\vert v\wedge w\vert,\text { for all }v,w\in TM.$$
	One of the first results for the graphical mean curvature flow in higher codimension was obtained by Tsui and Wang \cite{tsuiwang}, where they proved that each initial strictly area decreasing map $f:\mathbb{S}^m\to\mathbb{S}^n$ between unit spheres of dimensions $m,n\ge 2$ smoothly converges to a constant map under the flow.  This result has been generalized much further by other authors; see for instance \cite{ss2, ss0}. In \cite{ss2} we proved that the mean curvature flow smoothly deforms a strictly area decreasing map $f:M\to N$ into a constant one, if $M$ and $N$ are compact, the Ricci curvature $\Ric_M$ of $M$ and the sectional curvatures $\sigma_M$ and $\sigma_N$ of $M$ and $N$, respectively, satisfy
	\begin{equation}\label{oold}
	\sigma_M >-\sigma\quad\text{and}\quad \Ric_M \ge(m-1)\sigma\ge(m-1)\sigma_N\tag{O}
	\end{equation}
	for some positive constant $\sigma>0$, where $m$ is the dimension of $M$. Optimal results were obtained in \cite{ss0} for area decreasing maps between surfaces.
	
	We consider area decreasing maps $f:M\to N$, where $M$ is compact and $N$ is a complete
	surface $N$ with bounded geometry, that is the curvature of $N$ and its derivatives of all orders are uniformly bounded, and the injectivity radius is positive. In order to state our main results, we need to introduce some curvature conditions.
\begin{definition}
	Let $(M,\gm)$ be a Riemannian manifold of dimension $m>1$ and let $(N,\gn)$ be a Riemannian surface.
	For any pair of orthonormal vectors $v,w$ on $M$, the bi-Ricci curvature $\BR_{M}$ is given by
	$$
	{\BR}_M(v,w)=\Ric_M(v,v)+\Ric_M(w,w)-\sigma_M(v\wedge w),
	$$
	where $\Ric_M$ is the Ricci curvature and $\sigma_M$ the sectional curvature of $M$.
	\begin{enumerate}[\normalfont(A)]
		\item \label{curvature main} We say that the curvature condition \eqref{curvature main} holds, if the bi-Ricci curvature
		of $M$ is bounded from below by the sectional curvature of $N$, that is if
		${\BR}_M\ge{\sup}_N\sigma_N$.
		\smallskip
		\item\label{curvature b} We say that the curvature condition \eqref{curvature b} holds, if the Ricci curvature of $M$ is non-negative.
		\smallskip
		\item \label{curvature c} We say that the curvature condition \eqref{curvature c} holds, if the Ricci curvature  of $M$ is bounded from below by the sectional curvature of $N$, that is if
$\Ric_M\,\ge \,{\sup}_N\sigma_N$.
	\end{enumerate}
\end{definition}

The concept of bi-Ricci curvature was introduced by Shen and Ye \cite{shen}. Note that the condition \eqref{curvature c} implies \eqref{curvature b} if ${\sup}_N\sigma_N\ge 0$ and that \eqref{curvature b} implies \eqref{curvature c} if ${\sup}_N\sigma_N\le 0$. In particular, conditions \eqref{curvature b} and \eqref{curvature c} are equivalent if ${\sup}_N\sigma_N=0$. We will discuss these conditions in detail in 
Remark \ref{remark curvature}.
	
Our main results are  stated in Theorems \ref{main}, \ref{minimal}
and its corollaries which are presented in Section \ref{longtime}. Roughly speaking, in Theorem \ref{main} we obtain long-time existence
of the mean curvature flow of area decreasing maps
under the condition \eqref{curvature main} and convergence to minimal maps under the conditions 
\eqref{curvature main}, \eqref{curvature b}, and \eqref{curvature c}. The proof of Theorem \ref{main}
relies on an estimate for the mean curvature of the evolving submanifolds and a Bernstein type theorem for minimal graphs. The classification of these minimal
maps will be presented in Theorem \ref{minimal}.  The proofs of Theorems \ref{main} and \ref{minimal} are given in Section \ref{proofs}.

\section{Long-time existence and convergence of the flow}\label{longtime}
\noindent The main results for the mean curvature flow are stated below.

\begin{mythm}\label{main}
Let $(M,\gm)$ be a compact Riemannian manifold of dimension $m>1$ and let $(N,\gn)$ be a complete Riemannian surface of bounded geometry. Suppose $f_0:M\to N$ is strictly area decreasing.
\begin{enumerate}[\textnormal{(a)}]
\item[\textnormal{(a)}]
If the curvature condition \eqref{curvature main} holds, that is
$\BR_M\ge {\sup}_N\sigma_N$,
then the induced graphical mean curvature flow exists for $t\in[0,\infty)$, and the evolving maps $f_t:M\to N$ remain strictly area decreasing for all $t$.
\medskip
\item[\textnormal{(b)}]
If the curvature conditions \eqref{curvature main} and \eqref{curvature b} hold, that is 
$\BR_M\ge {\sup}_N\sigma_N$ and $\Ric_M\ge 0$,
then $\{f_t\}_{t\in[0,\infty)}$ is uniformly bounded in $C^1(M,N)$ and remains uniformly strictly area decreasing. 
\medskip
\item[\textnormal{(c)}]
If the curvature conditions \eqref{curvature main} and \eqref{curvature c} hold, that is
$\BR_M\ge {\sup}_N\sigma_N$ and $\Ric_M\ge{\sup}_N\sigma_N$,
then the mean curvature stays uniformly bounded. If $\{f_t\}_{t\in[0,\infty)}$ is uniformly bounded in $C^1(M,N)$, then $\{f_t\}_{t\in[0,\infty)}$ is uniformly bounded in $C^k(M,N)$, for all $k\ge 1$.
\medskip
\item[\textnormal{(d)}]
Suppose that the curvature conditions \eqref{curvature main}, \eqref{curvature b} and \eqref{curvature c} hold, that is we have
$\BR_M(v,w)\ge {\sup}_N\sigma_N$ and $\Ric_M(v,v)\ge\max\{0,{\sup}_N\sigma_N\}$.
Then we get the following results:
\smallskip
\begin{enumerate}[\normalfont(1)]
	\item $\{f_t\}_{t\in[0,\infty)}$ is uniformly bounded in $C^k(M,N)$, for all $k\ge 1$.
	\smallskip
	\item In the following cases $\{f_t\}_{t\in[0,\infty)}$ is uniformly bounded in $C^0(M,N)$:
	\begin{enumerate}[\normalfont(i)]
		\item $\Ric_M>0$.
		\item $N$ is compact.
		\item ${\sup}_N\sigma_N\le 0$ and $N$ is simply connected.
		\item ${\sup}_N\sigma_N\le 0$ and $N$ contains a totally convex subset $\mathscr{C}$; that is $\mathscr{C}$ contains any geodesic in $N$ with endpoints in $\mathscr{C}$.
		\item There exists $c\in\real{}$ and a smooth function $\psi:N\to\real{}$ such that $\psi$ is convex on the set $N^c:=\{y\in N:\psi(y)<c\}$, $\overline {N^c}$ is compact and  $f_0(M)\subset N^c$.
		\end{enumerate}
	\smallskip
	\item Under the assumption that the family $\{f_t\}_{t\in[0,\infty)}$ is uniformly bounded in $C^k(M,N)$, for all $k\ge 0$, the following holds:
	\smallskip
	\begin{enumerate}[\normalfont(i)]
		\item There exists a subsequence $\{f_{t_n}\}_{n\in\natural{}}$, $\lim_{n\to\infty}t_n=\infty$,  that smoothly converges to one of the minimal maps classified in Theorem \ref{minimal}. 
		\smallskip
		\item If there exists a subsequence $\{f_{t_n}\}_{n\in\natural{}}$ of the family $\{f_t\}_{t\in[0,\infty)}$ that converges in $C^0(M,N)$ to a constant map, then the whole flow $\{f_t\}_{t\in[0,\infty)}$ smoothly converges to this constant map.
		\smallskip
		\item If there exists a point $x\in M$ such that $\Ric_M(x)>0$, then the flow $\{f_t\}_{t\in[0,\infty)}$ smoothly converges to a constant map.
		\smallskip
		\item If $(M,\gm)$ and $(N,\gn)$ are real analytic, then the flow smoothly converges to one of the minimal maps classified in Theorem \ref{minimal}.
	\end{enumerate}
\end{enumerate}
\end{enumerate}
\end{mythm}
Let us discuss now some interesting corollaries of Theorem \ref{main}. 
\begin{mycor}\label{cor B}
	Let $M$ be a compact manifold with non-vanishing Euler characteristic $\chi(M)$, and let $N$ be a compact Riemann surface of genus bigger than one.
	\begin{enumerate}[\normalfont(a)]
		\item[\normalfont(a)] If $M$ is K\"ahler with vanishing first Chern class $c_1(M)$, then any smooth map $f:M\to N$ is smoothly null-homotopic.
		\smallskip
		\item[\normalfont(b)] More generally, the same result holds if $M$ admits a  metric of non-negative Ricci curvature.
	\end{enumerate}
\end{mycor}
\begin{proof}
	(a) If $M$ is a K\"ahler manifold with vanishing first Chern class, then by a famous theorem of Yau \cite{yau}, $M$ admits a Ricci flat K\"ahler metric and this case can be reduced to part (b).
	
	(b) Let $\gm$ be a metric of non-negative Ricci curvature on $M$. Since $N$ has genus bigger than one, we can endow $N$ with a complete Riemannian metric $\gind_N$ of constant negative curvature $\sigma_N$.
Since $\dim N=2$, the map $f:M\to N$ has at most two non-trivial singular values $\lambda\ge\mu$
with respect to the metrics $\gm$ and $g_N$.
For a constant $r>0$ define the new metric $\gind_r:=r^2\gind_N$. Then the sectional curvature $\sigma_r$ of $\gind_r$, and the singular values $\lambda_r$ and $\mu_r$ of $f$ with respect to $\gm$ and $\gind_r$ are given by
	$\lambda_r=r\lambda$, $\mu_r=r\mu$ and $\sigma_r=r^{-2}\sigma_N.$
	If we choose $r$ sufficiently small, then $f$ will be strictly area decreasing with respect to $\gm$, $\gind_r$ and $\sigma_r$ will be so small that all curvature conditions in \eqref{curvature main}, \eqref{curvature b} and \eqref{curvature c} are satisfied. Applying the mean curvature flow to the graph $\varGamma_f$ of $f$ in $(M,\gm)\times (N,\gn)$, Theorem \ref{main}(d) and Theorem \ref{minimal} imply that $f$ is homotopic to a constant map if $M$ does not have vanishing Euler characteristic.
\end{proof}
\begin{remark}
	Most of the Calabi-Yau manifolds have non-vanishing Euler characteristic. For example, the Euler number of K3-surfaces is $24$. The statement in (b)  cannot be extended to the case where $N$ is $\mathbb{S}^2$ or $\mathbb{T}^2$. Neither the Hopf fibration $f:\mathbb{S}^{3}\to\mathbb{S}^{2}$ nor the projections $\pi_{\mathbb{S}^2}:\mathbb{S}^1\times\mathbb{S}^2\to\mathbb{S}^2$,  $\pi_{\mathbb{T}^2}:\mathbb{S}^1\times\mathbb{T}^2\to\mathbb{T}^2$ are homotopic to a constant map or to a geodesic.
\end{remark}
\begin{remark}\label{remark curvature}
	It is easy to construct examples of long-time existence but no convergence. Take $M=\mathbb{S}^1\times\mathbb{S}^2$ with the standard product metric and for $N$ choose $\mathbb{S}^1\times\real{}$ with a rotationally symmetric metric of negative curvature; see  Figure \ref{areaAB}(a). Then $\Ric_M\ge 0\ge\sigma_N$ and the condition \eqref{curvature main} is satisfied. Fix $z_0\in\real{}$, let $c_0:\mathbb{S}^1\to N$ be the circle $c_0(s)=(s,z_0)$ and define $f_0(s,p):=c_0(s)$,  $(s,p)\in \mathbb{S}^1\times\mathbb{S}^2$. Clearly, $f_0$ is strictly area decreasing. The solution $f_t$ to the mean curvature flow will be of the form $f_t(s,p)=(s,z(t))$, where $z:[0,\infty)\to\real{}$ is a smooth function that becomes unbounded when $t\to\infty$. All conditions in Theorem \ref{main}(d) are satisfied, except those in (2) guaranteeing $C^0$-bounds. However, the solution is uniformly bounded in $C^k(M,N)$ for all $k\ge 1$ and the mean curvature tends to zero in $L^2$ for $t\to\infty$.
	Nevertheless, there exist rotationally symmetric hyperbolic metrics on the cylinder for which we can apply Theorem \ref{main}(d). For example, the closed geodesic $\mathscr{C}$ on the one-sheet hyperboloid depicted in Figure \ref{areaAB}(b) is totally convex.
\begin{figure}[!ht]
	\includegraphics[width=.48\textwidth]{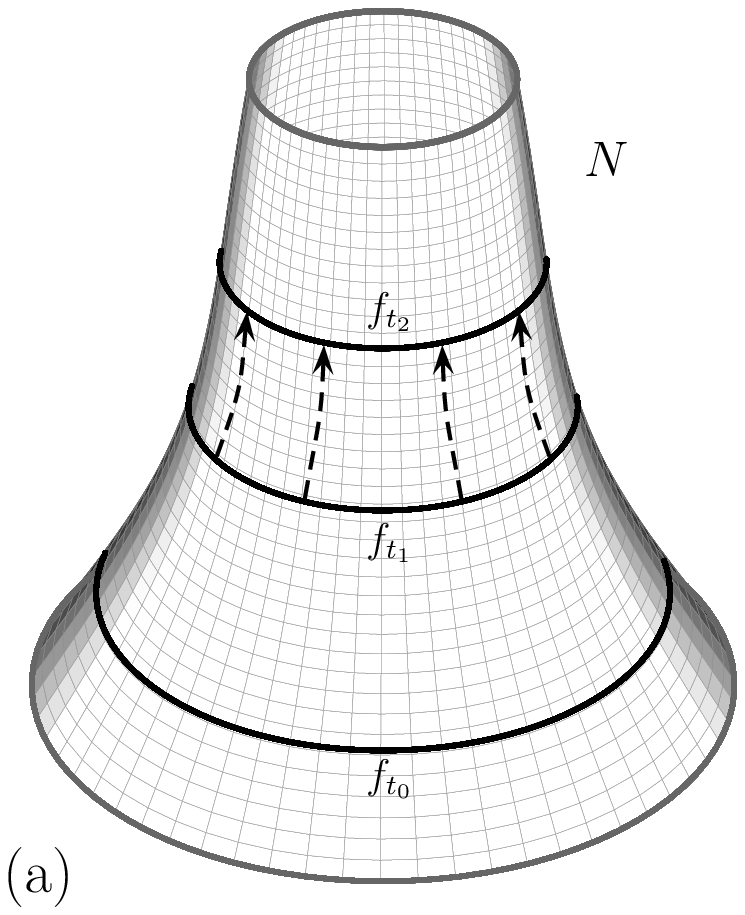}
	\includegraphics[width=.48\textwidth]{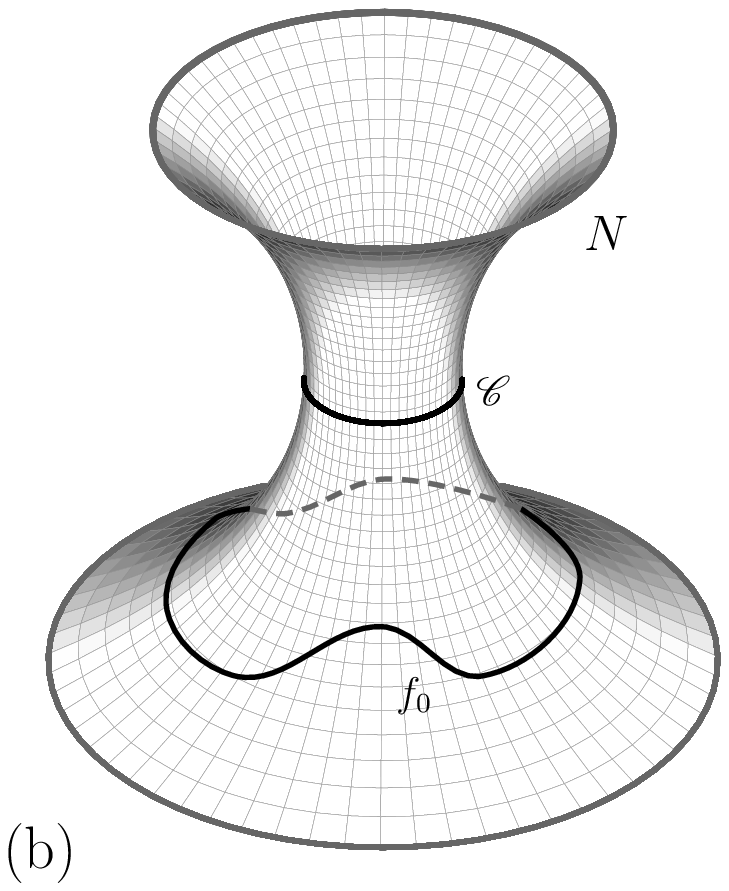}
	\begin{minipage}[t]{1.08\textwidth}
		\caption{(a) Negatively curved surface without closed geodesics. (b) Negatively curved
		surface with a totally convex subset.}
		\label{areaAB}%
	\end{minipage}
\end{figure}
\end{remark}
\begin{remark}\label{remark 1}
	We add some remarks concerning the curvature conditions and the results of  Theorem \ref{main}.
	\begin{enumerate}[(a)]
		\item In dimensions $m=2,3$, the curvature condition \eqref{curvature main} is
		 equivalent to
		$\Scal_M\ge\Scal_N,$ where these are the scalar curvatures of $M$ and $N$.
		Therefore, we recover the main results obtained in \cite{ss0}.
		For $m=4$, \eqref{curvature main} is equivalent to 
		$\operatorname{Scal}_M-2\,{\sigma}_M(v\wedge w)\ge \Scal_N,$ for all $v,w\in TM.$
		\smallskip
		\item
		If $M$ and $N$ satisfy \eqref{curvature main}, then by taking traces 		
		 at each point $x\in M$, the scalar and the Ricci curvatures of $M$ can be estimated by
		\begin{eqnarray}
		(m-3)&\!\!\!\operatorname{Ric}_M(v,v)\!\!\!\!&+\operatorname{Scal}_M\ge (m-1){\sup}_N\sigma_N,\label{curvcond2b}\\
		&\!\!\!{\Scal}_M\!\!\!\!\!\!\!\!\!\!\!\!\!\!&\ge\frac{m(m-1)}{2m-3}{\sup}_N\sigma_N,\label{curvcond3}
		\end{eqnarray}
		and, for $m\neq 3$, equality occurs if and only if at $x$ all sectional curvatures of $M$ are equal.
Thus, if $m\neq 3$ and at each $x\in M$ there exist at least two distinct sectional curvatures, then \eqref{curvature main} can only be satisfied as a strict inequality.
		\medskip
		\item
		The results in \cite{ss2,ss1} were obtained under the assumption \eqref{oold}. It turns out that \eqref{oold} implies \eqref{curvature main} and, in particular, in this case \eqref{curvature main} becomes even strict when $m>2$. Indeed, if $m=2$ the conclusion follows from $\Ric_M(v,v)=\Ric_M(w,w)$ for any $v,w\in TM$. In case $m>2$, it suffices to check this for an orthonormal frame $\{\alpha_1,\dots,\alpha_m\}$ for which the Ricci tensor becomes diagonal. Then, for any $i\neq j$, we get
		\begin{eqnarray*}
		\hspace{+15pt}\BR_M(\alpha_i,\alpha_j)&\!\!\!=\!\!\!&\operatorname{Ric}_M(\alpha_i,\alpha_i)+{\sum}_{k\neq j,i}\sigma_M(\alpha_k\wedge \alpha_j)\\
		&\!\!\!>\!\!\!&(m-1)\sigma-(m-2)\sigma=\sigma\ge{\sup}_N\sigma_N.
		\end{eqnarray*}
		However, \eqref{curvature main} does not imply \eqref{oold}, hence condition \eqref{curvature main} is more general than \eqref{oold}. To obtain a better picture, let us assume that the sectional curvatures of $(M,\gm)$ are all constant to $\sigma_M$ and that the curvature of $N$ is given by a constant $\sigma_N$. The curvature condition \eqref{oold} of \cite{ss2} is then equivalent to
		$\sigma_M\ge \sigma_N$ and $\sigma_M>0.$
		When the sectional curvatures are constant, \eqref{curvature main}, \eqref{curvature b} and \eqref{curvature c} are equivalent to (in this order):
		$$
			(2m-3)\sigma_M\ge\sigma_N,\quad\sigma_M\ge0\quad\text{and}\quad
			(m-1)\sigma_M\ge\sigma_N.
		$$
		Therefore, the results in
	Theorem \ref{main} are stronger than those in \cite{ss2}.
	\end{enumerate}
\end{remark}
Given a map $f:\mathbb{S}^m\to \mathbb{S}^n$ between
unit spheres with singular values
$\lambda_1\ge\lambda_2\ge\cdots\ge\lambda_m$,
the number $\operatorname{Dil}_2(f)=\max\lambda_{1}\lambda_2$, is called the 2-dilation of $f$.
An interesting question is to determine when such a map
is homotopically trivial. In this direction, we obtain the following result.
\begin{mycor}\label{cor cont}
	For the standard unit spheres $(\mathbb{S}^m,\gind_{\mathbb{S}^m})$ and $(\mathbb{S}^2,\gind_{\mathbb{S}^2})$ let us define
	$$\mathcal{A}_{m-1}:=\{f\in C^\infty(\mathbb{S}^m,\mathbb{S}^2):\operatorname{Dil}_2(f)< m-1\}.$$
	Then for $m>1$ and for any $f_0\in\mathcal{A}_{m-1}$ there exists a smooth homotopy $\{f_t\}_{t\in[0,\infty)}\subset\mathcal{A}_{m-1}$ deforming $f_0$ into a constant map. This homotopy can be given by the mean curvature flow of $f_0$ as a map between $(\mathbb{S}^m,\gind_{\mathbb{S}^m})$ and the scaled $2$-sphere $(\mathbb{S}^2,(m-1)^{-1}\gind_{\mathbb{S}^2})$. In particular, $\mathcal{A}_{m-1}$ is smoothly contractible.
\end{mycor}
\begin{proof}
	Maps in $\mathcal{A}_{m-1}$ are strictly area decreasing maps from $(\mathbb{S}^m,\gind_{\mathbb{S}^m})$ to $(\mathbb{S}^2,(m-1)^{-1}\gind_{\mathbb{S}^2})$. The sectional curvature of $\gn:=(m-1)^{-1}\gind_{\mathbb{S}^{2}}$
	is $m-1$ and the result follows from Remark \ref{remark 1}(e), because in this case the curvature conditions in Theorem \ref{main} are equivalent to $m-1\ge\sigma_N$.
\end{proof}

\begin{remark}
It is well-known that the homotopy groups $\pi_m(\mathbb{S}^2)$ are non-trivial for $m\ge 2$ and are finite for $m\ge 4$;
see \cite{berrick,curtis,gray}.
Consequently, in Corollary \ref{cor cont}, we cannot increase the upper bound for $\operatorname{Dil}_2(f)$ arbitrarily
without losing the contractibility of the corresponding set
	$$\mathcal{A}_{m,c}:=\{f\in C^\infty(\mathbb{S}^m,\mathbb{S}^2):\operatorname{Dil}_2(f)< c\}.$$
A natural problem arises; to determine the number
	$$\mathfrak{c}_m:=\sup\{c>0:\mathcal{A}_{m,c}\text{ is smoothly contractible}\}.$$
The Hopf fibration $f:\mathbb{S}^3\to\mathbb{S}^2$ has constant singular values
$\lambda_1=\lambda_2=2$ and $\lambda_3=0$.
Moreover, it is minimal, but not totally geodesic, and not homotopic to a constant map; see \cite[Remark 1]{markellos}. Hence, from Corollary \ref{cor cont} we see that
	$$2\le\mathfrak{c}_3\le 4 \quad \text{ and }\quad m-1\le\mathfrak{c}_m<\infty,$$
	for $m>2$. Since the identity map $\Id:\mathbb{S}^2\to\mathbb{S}^2$ is not homotopic to the
	constant map, we have that $\mathfrak{c}_2=1$.
\end{remark}
	In dimension three the results in Theorem \ref{main} can be summarized in the following corollary.
	
\begin{mycor}\label{cor C}
	Let $(M,\gm)$ be a compact $3$-manifold and let $(N,\gn)$ be a complete surface of bounded geometry that satisfy the curvature condition
	$$\Ric_M\ge\max\{0,{\sup}_N\sigma_N\}.$$
	Then \eqref{curvature main}, \eqref{curvature b} and \eqref{curvature c} in Theorem {\normalfont\ref{main}} are satisfied and for any 
	strictly area decreasing initial map $f_0:M\to N$ the results in Theorem {\normalfont\ref{main}}{\normalfont(d)} apply.
	\end{mycor}
\begin{proof}
	The conditions \eqref{curvature b} and \eqref{curvature c} hold  by assumption. Since $m=3$, the  condition \eqref{curvature main} is equivalent to
	$\operatorname{Scal}_M\ge 2\,{\sup}_N\sigma_N.$
	We distinguish two cases.
	\begin{enumerate}[(i)]
		\item ${\sup}_N\sigma_N\ge 0$. In this case \eqref{curvature c} $ \Rightarrow\operatorname{Scal}_M\ge 3\,{\sup}_N\sigma_N\ge 2\,{\sup}_N\sigma_N$.
		\smallskip
		\item ${\sup}_N\sigma_N\le 0$. In this case \eqref{curvature b} $ \Rightarrow\operatorname{Scal}_M\ge 0\ge{\sup}_N\sigma_N\ge 2\,{\sup}_N\sigma_N$.
	\end{enumerate}
	Therefore, the curvature condition \eqref{curvature main}  holds and Theorem \ref{main} applies.
\end{proof}
The next corollary follows from the K\"unneth formula and the fact that compact manifolds with positive Ricci curvature do not admit non-trivial harmonic $1$-forms. We give a proof using mean curvature flow.
\begin{mycor}\label{cor D}
	Let $M=L\times N$ be the product of a compact manifold $L$ and a compact surface $N$ of genus 
	bigger than one. Then $M$ does not admit any Riemannian metric of positive Ricci curvature.
\end{mycor}
\begin{proof}
	The projection $\pi_N:L\times N\to N$ is not homotopic to a constant map. If $L\times N$ admits a metric of positive Ricci curvature, then we can equip $N$ with a metric of sufficiently negative constant curvature such that $\pi_N$ becomes strictly area decreasing and such that the curvature conditions \eqref{curvature main}, \eqref{curvature b} and \eqref{curvature c} hold. Theorem \ref{main} implies that $\pi_N$ can be deformed into a constant map by mean curvature flow. This is a contradiction.
\end{proof}

We state the classification of the limits in Theorem \ref{main}.
If $\dim N=2$, then $f:M\to N$ has at most two non-trivial singular values
$\lambda\ge\mu$.
\begin{mythm}\label{minimal}
	Let $(M,\gm)$ be a compact Riemannian manifold of dimension $m>1$ and let $(N,\gn)$ be a complete Riemannian surface such that \eqref{curvature main} and \eqref{curvature b} hold, that is we have
	\begin{gather*}
	\BR_M\ge {\sup}_N\sigma_N\quad\text{and}\quad	\Ric_M\ge 0.
	\end{gather*}
	Let $f:M\to N$ be a strictly area decreasing minimal map. Then $f$ is totally geodesic, the rank $\rank(df)$ of $df$ and the singular values $\lambda$ and $\mu$ of $f$ are constant. If $\rank(df)=0$, then $f$ is constant and $\lambda=\mu=0$. Otherwise, $\,\rank(df)>0$ and $f:M\to f(M)$ is a submersion. Each fiber $K_y$,
	$y\in f(M)$, is a compact embedded totally geodesic submanifold that is isometric to a manifold $(K,\gind_K)$ of non-negative Ricci curvature that does not depend on $y$.  The horizontal integral submanifolds are complete totally geodesic submanifolds in $M$ that intersect the fibers orthogonally. $(M,\gm)$ is locally isometric to a product $(L\times K,\gind_L\times \gind_K)$. The Euler characteristic $\chi(M)$ of $M$ vanishes, and, at each $x\in M$, the kernel of the Ricci operator is non-trivial. More precisely:
		\smallskip
		\begin{enumerate}[\normalfont(a)]
			\item $\rank(df)=1$. Then $\mu=0$, $\lambda>0$. Moreover, $\upgamma:=f(M)$ is a closed geodesic in $N$. The horizontal leaves are geodesics orthogonal to the fibers and $f:(M,\gm)\to(\upgamma,\lambda^{-2}\gind_{\upgamma})$ is a Riemannian submersion, where $\gind_{\upgamma}$ denotes the metric on $\upgamma$ as a submanifold in $(N,\gn)$.
			\smallskip
			\item $\rank(df)=2$. Then $\lambda,\mu>0$, $f(M)=N$, and $N$ is diffeomorphic to a torus $\mathbb{T}^2$ or a Klein bottle $\mathbb{T}^2/\mathbb{Z}_2$. The metric $\gn$ and the metrics on the horizontal leaves are flat. Additionally,  $f:(M,\gm)\to(N,\lambda^{-2}\gn)$ is a Riemannian submersion, if $\lambda=\mu$.
		\end{enumerate}
\end{mythm}

\begin{mycor}\label{cor b}
	If,	 in addition to the assumptions made in Theorem {\normalfont\ref{minimal}}, there exists a point $x\in M$ with $\Ric_M(x)>0$, then strictly area decreasing minimal maps $f:M\to N$ are constant.
\end{mycor}
\begin{proof}
	If $\Ric_M(x)>0$ at some point $x\in M$, then $\rank(df)>0$ in Theorem F is impossible, since this part requires the kernel of the Ricci operator to be non-trivial at each point.
\end{proof}

\section{Geometry of graphs}\label{graphs}
\noindent In this section, we follow the notations of our previous papers \cite{ss0,ss1, ss2} and recall some basic 
facts related to the geometry of graphical submanifolds.
\smallskip
\subsection{Fundamental forms and connections}~ 

	The product $M\times N$ will be regarded as a Riemannian manifold equipped with the metric
	$\gk=\langle\cdot\,,\cdot\rangle:=\gm\times \gn.$
 The graph $\varGamma_f$ is parametrized by $F:=\operatorname{Id}_{M}\times f$, where $\operatorname{Id}_{M}$ is the identity map of $M$. The metric on $M$ induced by $F$ will be denoted by
	$\gind:=F^*\gk$ and will be called the {\em graphical metric}. The Levi-Civita connection of $\gind$ is denoted by $\nabla$, the curvature tensor by $\rind$ and the Ricci curvature by $\Ric$.

	Denote by $\pi_{M}:M\times N\to M$ and $\pi_{N}:M\times N\to N$ the two natural projections. The metric tensors $\gm,\gn,\gk$ and $\gind$ are related by
		$$\gk:=\pi_M^*\gm+\pi_N^*\gn\quad\text{and}\quad\gind:=F^*\gk=\gm+f^*\gn.$$
	As in \cite{ss0,ss1, ss2}, let us define the symmetric $2$-tensors
	\begin{equation*}
		\sk:=\pi_M^*\gm-\pi_{N}^*\gn\quad\text{and}\quad
		\sind:=F^*\sk=\gm-f^*\gn\,.\label{met4}
	\end{equation*}
	The \textit{second fundamental form} of $F$ is denoted by the letter $A$. In terms of the connections $\nabla^F$
	and $\nabla^f$ of the pull-back bundles $F^*T(M\times N)$ and $f^*TN$, respectivley, we have
	\begin{eqnarray}
	A(v,w)&=&\nabla_v^F\bigl(dF(w)\bigr)-dF(\nabla_vw)\nonumber\\
	&=&\left(\nabla_v^{\gm}w-\nabla_vw,\nabla_v^f\bigl(df(w)\bigr)-df(\nabla_vw)\right),\label{sec formula}
	\end{eqnarray}
	where $v,w$ are arbitrary smooth vector fields on $M$. In the sequel, we will denote all full connections on bundles over $M$ which are induced by the Levi-Civita connection of $\gk$ via $F:M\to M\times N$ by the same letter $\nabla$.
	
	If $\xi$ is a normal vector of the graph, then the symmetric bilinear form $A^{\xi}$, given by
	$$A^{\xi}(v,w):=\langle A(v,w),\xi\rangle,$$
	will be called the \textit{second fundamental form with respect to the normal $\xi$}. The \textit{mean curvature vector field} of the graph $\varGamma_f$ is the trace of $A$ with respect to the graphical metric $\gind$, that is
	$$H:={\trace}_{\gind}A$$
	and $H$ is a section in the normal bundle $T^\perp M$. The graph $\varGamma_f$, and likewise the map $f$, are called \textit{minimal} if $H$ vanishes identically.

	Throughout this paper, we will use latin indices to indicate components of tensors with respect to frames in the tangent bundle that are orthonormal with respect to $\gind$. For example, if $\{e_1,\dots,e_m\}$ is a local orthonormal frame of the tangent bundle and $\xi$ is a local vector field in the normal bundle of $M$, then
	$$A_{ij}=A(e_i,e_j)\quad\text{and}\quad A^{\xi}_{ij}=\langle A(e_i,e_j),\xi\rangle.$$

\subsection{Singular value decomposition of maps in codimension two}\label{frames}~

	Fix a point $x\in M$ and let $\lambda^2_{1}\ge\displaystyle{\dots}\ge\lambda^2_{m}$ denote the eigenvalues of $f^{*}\gn$ at $x$ with respect to $\gm$. The corresponding values $\lambda_i\ge  0$, $i\in\{1,\dots,m\}$,
	are the singular values of the differential $\df$ of $f$ at the point $x$. The singular values are Lipschitz continuous functions on $M$.

	Suppose that $M$ has dimension $m>1$ and that $N$ is a Riemannian surface. In this case, there exist at most two non-vanishing singular values, which we denote for simplicity by $\lambda:=\lambda_1$ and $\mu:=\lambda_2.$ At each fixed point $x\in M$, one may consider an orthonormal basis $\{\alpha_{1},\dots,\alpha_{m}\}$ of $T_xM$ with respect to $\gm$ that diagonalizes $f^*\gn$. Therefore, at $x$ we have 
	$$\left(f^*\gn(\alpha_i,\alpha_j)\right)_{i,j}=\operatorname{diag}\bigl(\lambda^2,\mu^2,0,\dots,0\bigr).$$
	In addition, at $f(x)$ we may consider an orthonormal basis $\{\beta_{1},\beta_{2}\}$ with respect to $\gn$ such that
	$$\df(\alpha_{1})=\lambda \beta_{1},\quad \df(\alpha_{2})=\mu \beta_{2}\quad\text{and}\quad\df(\alpha_{i})=0, \,\,\text{for }i\ge 3.$$
We then define another basis $\{e_1,\dots,e_m\}$ of $T_xM$ and a basis $\{\xi,\eta\}$ of $T_x^\perp M$ in terms of the singular values, namely
	\begin{equation*}
	e_1:=\frac{\alpha_1}{\sqrt{1+\lambda^2}}, \quad e_2:=\frac{\alpha_2}{\sqrt{1+\mu^2}},\quad e_{i}:=\alpha_i,\,\,\text{for }i\ge 3,\label{tangent}
	\end{equation*}
	and 
	\begin{equation*}
	\xi:=\frac{-\lambda\alpha_1\oplus\beta_1}{\sqrt{1+\lambda^2}},\quad
	\eta:=\frac{-\mu\alpha_2\oplus\beta_2}{\sqrt{1+\mu^2}}.\label{normal}
	\end{equation*}
	The frame $\{e_1,\dots,e_m\}$ defines an orthonormal basis of $T_xM$ with respect to the induced graphical metric $\gind$, and $\{\xi,\eta\}$ forms an orthonormal basis of $T_x^\perp M$ at $F(x)$. The pull-back $\sind=F^*\sk$ to $TM$ satisfies
	\begin{equation}\label{sind}
	\left(\sind(e_{i},e_{j})\right)_{i,j}=\operatorname{diag}\left(\frac{1-\lambda^{2}}{1+\lambda^{2}},\frac{1-\mu^{2}}{1+\mu^{2}},1,\dots,1\right).
	\end{equation}
	The restriction $\sind^\perp$ of $\sk$ to the normal bundle of $\varGamma_f$ satisfies the identities
	\begin{eqnarray}
	\sind^{\perp}(\xi,\xi)=-\frac{1-\lambda^{2}}{1+\lambda^{2}},\,\, \sind^{\perp}(\eta,\eta)=-\frac{1-\mu^{2}}{1+\mu^{2}}\,\,\text{and}\,\, \sind^{\perp}(\xi,\eta)=0.\label{normal2}
	\end{eqnarray}
Define the quantities $\tind_{11}=\sk(\dF(e_{1}),\xi)$ and $\tind_{22}=\sk(\dF(e_{2}),\eta)$,
which represent the mixed terms of $\sk$.
Note that
\begin{equation}\label{mixed1}
\tind_{11}:=-\frac{2\lambda}{1+\lambda^{2}}\quad\text{and}\quad
\tind_{22}:=-\frac{2\mu}{1+\mu^{2}}.
\end{equation}
A map $f$ is strictly area decreasing if ${\lambda\mu}<1$. Consider $\pind:M\to\real{}$ given by
	$$\pind:=\operatorname{tr}_{\gind}\sind+2-m=\sind_{11}+\sind_{22}=\frac{2(1-\lambda^2\mu^2)}{(1+\lambda^2)(1+\mu^2)}.$$
	In codimension two, the map $f$ is strictly area decreasing if and only if $\pind>0$.

\section{Estimates for the graphical mean curvature flow}\label{estimatesmeancurvature}
	\noindent Let $f:M\to N$ be a smooth map between two Riemannian manifolds and let
	$F_0:=\operatorname{Id}_M\times f:M\to\varGamma_f\subset M\times N.$
	We deform the graph $\varGamma_f$ by the mean curvature flow in $M\times N$, that is we consider the family of immersions $F:M\times [0,T)\to M\times N$ satisfying the evolution equation
	\begin{gather}\label{mcf}
	\frac{dF}{dt}(x,t)=H(x,t),\quad
	F(x,0)=F_0(x).\tag{MCF}
	\end{gather}
	where $(x,t)\in M\times[0,T)$, $H(x,t)$ is the mean curvature vector field at $x\in M$ of $F_t :M\to M\times N$, $F_t(\,\cdot\,) := F(\,\cdot\,,t)$, and where $T$ denotes the maximal time of existence of a smooth solution of \eqref{mcf}.
	\smallskip

\subsection{First order estimates for area decreasing maps}~

	\noindent To investigate under which conditions the area decreasing property is preserved under the flow, we compute the evolution equation of the function
	$\pind=\operatorname{tr}_{\gind}\sind+2-m.$
	We have the following result for graphs in codimension 2.
	\begin{lemma}\label{evolpind}
	The function $\pind$ satisfies the evolution equation
	\begin{eqnarray}\label{p8}
	\big(\nabla_{\dt}-\Delta\big)\pind&\!\!\!\!=\!\!\!\!&2\pind\vert A\vert^2+2\sum_{k=1, i=3}^m\vert A^\xi_{ki}\vert ^2(1-\sind_{11})+2\sum_{k=1, i=3}^m\vert A^\eta_{ki}\vert ^2(1-\sind_{22})\nonumber\\
	&&+\frac{1}{2\pind}\Big(4\sum_{k=1}^m\vert A^\xi_{1k}\tind_{22}+A^\eta_{2k}\tind_{11}\vert^2-\vert\nabla\pind\vert^2\Big)+\mathcal{Q},
	\end{eqnarray}
	where $\mathcal{Q}$ is the first order quantity given by
	\begin{eqnarray}
	\mathcal{Q}&:=&\frac{2\lambda^2\mu^2(2+\pind)}{(1+\lambda^2)(1+\mu^2)} 
	\big(\BR_M(\alpha_1,\alpha_2)-\sigma_N\big)\\
	&&+\frac{2\lambda^2\pind}{(1+\lambda^2)(1+\mu^2)}\operatorname{Ric}_M(\alpha_1,\alpha_1)+\frac{2\mu^2\pind}{(1+\lambda^2)(1+\mu^2)}\operatorname{Ric}_M(\alpha_2,\alpha_2)\nonumber,\label{p8a}
	\end{eqnarray}
	and $\{\alpha_1,\dots,\alpha_m\}$, $\{e_1,\dots,e_m\}$, $\{\xi,\eta\}$ are the special bases arising from the singular value decomposition defined in subsection \ref{frames}.
	\end{lemma}
\begin{proof}
	To derive the evolution equation of $\pind$, we use the evolution equations of $\gind$ and $\sind$ that were derived in \cite{ss2}. Recall that
	\begin{equation}\label{evolgind}
		\bigl(\nabla_{\dt}\gind\big)(v,w)=-2A^H(v,w),
	\end{equation}
	and
	\begin{eqnarray}\label{evolsind}
		&&\big(\nabla_{\dt}\sind\big)(v,w)=\big(\Delta\sind\big)(v,w)-\sind(\operatorname{Ric}v,w)-
		\sind(\operatorname{Ric}w,v)\\
		&&\quad\quad\quad-2\sum_{k=1}^m\sind^\perp(A(e_k,v),A(e_k,w))+2\sum_{k=1}^m\big(\rm-f_{t}^*\rn\big)(e_k,v,e_k,w),\nonumber
	\end{eqnarray}
	for any $v,w\in TM$. Combining \eqref{evolsind} with the trace of the Gau{\ss} equation
	\begin{eqnarray*}
	\operatorname{Ric}(v,w)&=&\sum_{k=1}^m(\rm+f^*\rn)(e_k,v,e_k,w)\\
	&&+\sum_{k=1}^m\langle A(e_k,e_k),A(v,w)\rangle-\sum_{k=1}^m\langle A(v,e_k),A(w,e_k)\rangle
	\end{eqnarray*}
	we obtain
	\begin{eqnarray}
	\big(\nabla_{\dt}-\Delta\big)\pind&=&2\sum_{k,l=1}^m\big(\rm-f_{t}^*\rn\big)_{klkl}
		-2\sum_{k,l=1}^m\big(\rm+f_{t}^*\rn\big)_{klkl}\sind_{ll}\nonumber\\
	&&+2\left(\phantom{\sum\hspace{-18pt}}\right.\underbrace{\sum_{k,l=1}^m\langle A_{kl},A_{kl}\rangle\sind_{ll}
	-\sum_{k,l=1}^m\sind^\perp(A_{kl},A_{kl})}_{=:\mathcal{A}}\left.\phantom{\sum\hspace{-20pt}}\right).\label{p1}
	\end{eqnarray}
	In codimension two we are able to simplify this equation further. We start with the terms on the right hand side of the second line in \eqref{p1}. Since
	$$\sind_{ll}=1,\,\text{for}\, l\ge 3,\, \sind^\perp(\xi,\xi)=-\sind(e_1,e_1),\, \sind^\perp(\eta,\eta)=-\sind(e_2,e_2),\,\sind^\perp(\xi,\eta)=0,$$
	we get
	\begin{eqnarray*}
	\mathcal{A}&=&\sum_{k,l=1}^m\langle A_{kl},A_{kl}\rangle\sind_{ll}
	-\sum_{k,l=1}^m\sind^\perp(A_{kl},A_{kl})\\
	&=&\sum_{k=1}^m\vert A_{k1}\vert ^2\sind_{11}+\sum_{k=1}^m\vert A_{k2}\vert ^2\sind_{22}+\sum_{k=1, i=3}^m\vert A_{ki}\vert ^2+\vert A^\xi\vert^2\sind_{11}+\vert A^\eta\vert^2\sind_{22}\\
	&=&\pind\vert A\vert^2+\sum_{k=1, i=3}^m\big(\vert A^\xi_{ki}\vert ^2+\vert A^\eta_{ki}\vert ^2\big)\\
	&&+\Big(\sum_{k=1}^m\vert A_{k1}\vert ^2-\vert A^\eta\vert^2\Big)\sind_{11}+\Big(\sum_{k=1}^m\vert A_{k2}\vert ^2-\vert A^\xi\vert^2\Big)\sind_{22}.
	\end{eqnarray*}
	On the other hand 
	$$
	\sum_{k=1}^m\vert A_{k1}\vert ^2-\vert A^\eta\vert^2=\sum_{k=1}^m\big(\vert A^\xi_{k1}\vert^2-\vert A^\eta_{k2}\vert^2\big)-\sum_{k=1, i=3}^m\vert A^\eta_{ki}\vert ^2
	$$
	and
	$$\sum_{k=1}^m\vert A_{k2}\vert ^2-\vert A^\xi\vert^2=\sum_{k=1}^m\big(\vert 
	A^\eta_{k2}\vert^2-\vert A^\xi_{k1}\vert^2\big)-\sum_{k=1, i=3}^m\vert A^\xi_{ki}\vert ^2.$$
	Consequently,
	\begin{eqnarray}
	\mathcal{A}&=&\pind\vert A\vert^2
	+\sum_{k=1, i=3}^m\vert A^\xi_{ki}\vert ^2(1-\sind_{11})+\sum_{k=1, i=3}^m\vert A^\eta_{ki}\vert ^2(1-\sind_{22})\nonumber\\
	&&\quad\quad\,\,\,\,-\underbrace{\sum_{k=1}^m\big(\vert A^\xi_{1k}\vert^2-\vert A^\eta_{2k}\vert^2\big)(\sind_{22}-\sind_{11})}_{=:\mathcal{B}}.\label{p3}
	\end{eqnarray}
	We want to express $\mathcal{B}$ in terms of $\vert\nabla \pind\vert^2$, therefore we need a different expression for $\vert\nabla \pind\vert^2$. Since
	\begin{equation}\label{gradpind}
	(\nabla_{e_k}\sind)(v,v)=2\sk(A(e_k,v),\dF(v))
	\end{equation}
	we get
	\begin{eqnarray*}
	\nabla_{e_k}\pind
	&=&2\sum_{i=1}^m\sk(A(e_k,e_i),\dF(e_i))\\
	&=&2\sum_{i=1}^m\sk(\xi,\dF(e_i))A^\xi_{ik}+2\sum_{i=1}^m\sk(\eta,\dF(e_i))A^\eta_{ik},
	\end{eqnarray*}
	from where we deduce that
	\begin{equation*}
	\nabla_{e_k}\pind=2A^\xi_{1k}\tind_{11}+2A^\eta_{2k}\tind_{22}.
	\end{equation*}
	Recalling that $\sind_{ll}^2+\tind_{ll}^2=1$, for $l\in\{1,2\}$, we obtain for
	$\vert\nabla\pind\vert^2$ the following
	\begin{eqnarray}
	&&\hspace{-20pt}4\pind \mathcal{B}\!=\!4\pind\!\sum_{k=1}^m\big(\vert A^\xi_{1k}\vert^2\!-\!\vert A^\eta_{2k}\vert^2\big)(\sind_{22}-\sind_{11})
	=4\!\sum_{k=1}^m\big(\vert A^\xi_{1k}\vert^2\!-\!\vert A^\eta_{2k}\vert^2\big)(\sind_{22}^2-\sind_{11}^2)\nonumber\\
	&&\,\,\,=\vert\nabla\pind\vert^2-4\sum_{k=1}^m\vert A^\xi_{1k}\tind_{22}+A^\eta_{2k}\tind_{11}\vert^2\label{p4}.
	\end{eqnarray}
	Combining \eqref{p1}--\eqref{p4}, we derive that at points where $\pind>0$ it holds
	\begin{eqnarray}
	\big(\nabla_{\dt}-\Delta\big)\pind&\!\!\!=\!\!\!&2\pind\vert A\vert^2+2\!\!\!\sum_{k=1, i=3}^m\vert A^\xi_{ki}\vert ^2(1-\sind_{11})+2\!\!\!\sum_{k=1, i=3}^m\vert A^\eta_{ki}\vert ^2(1-\sind_{22})\nonumber\\
	&&+\frac{1}{2\pind}\Big(4\sum_{k=1}^m\vert A^\xi_{1k}\tind_{22}+A^\eta_{2k}\tind_{11}\vert^2-\vert\nabla\pind\vert^2\Big)\nonumber\\
	&&+\underbrace{2\sum_{k,l=1}^m(\rm)_{klkl}(1-\sind_{ll})}_{=:\mathcal{C}_1}-\underbrace{2\sum_{k,l=1}^m(f_{t}^*\rn)_{klkl}(1+\sind_{ll})}_{=:\mathcal{C}_2}.\label{p5}
	\end{eqnarray}
	Since $\sind_{ll}=1$ and $\lambda_l=0$ for $l\ge 3$, the first term $\mathcal{C}_1$ in the last line of \eqref{p5} simplifies to
	\begin{eqnarray*}
	\mathcal{C}_1&=&2\sum_{k,l=1}^m(\rm)_{klkl}(1-\sind_{ll})
	=2\sum_{k,l=1}^m\rm(e_k,e_l,e_k,e_l)(1-\sind(e_l,e_l))\\
	&=&\frac{4\lambda^2}{1+\lambda^2}\sum_{k=1}^m\rm(e_k,e_1,e_k,e_1)+\frac{4\mu^2}{1+\mu^2}\sum_{k=1}^m\rm(e_k,e_2,e_k,e_2)\\
	&=&\frac{4\lambda^2}{(1+\lambda^2)^2}\sum_{k=1}^m\Big(1-\frac{\lambda_k^2}{1+\lambda_k^2}\Big)\sigma_M(\alpha_1\wedge\alpha_k)\\
	&&+\frac{4\mu^2}{(1+\mu^2)^2}\sum_{k=1}^m\Big(1-\frac{\lambda_k^2}{1+\lambda_k^2}\Big)\sigma_M(\alpha_2\wedge\alpha_k)\\
	&=&\frac{4\lambda^2}{(1+\lambda^2)^2}\operatorname{Ric}_M(\alpha_1,\alpha_1)+\frac{4\mu^2}{(1+\mu^2)^2}\operatorname{Ric}_M(\alpha_2,\alpha_2)\\
	&&-\frac{4\lambda^2\mu^2}{(1+\lambda^2)(1+\mu^2)}\left(\frac{1}{1+\lambda^2}+\frac{1}{1+\mu^2}\right)\sigma_M(\alpha_1\wedge\alpha_2)\\
	\end{eqnarray*}
	Hence
	\begin{eqnarray*}
		\mathcal{C}_1&=&\frac{4\lambda^2}{(1+\lambda^2)^2}\operatorname{Ric}_M(\alpha_1,\alpha_1)+\frac{4\mu^2}{(1+\mu^2)^2}\operatorname{Ric}_M(\alpha_2,\alpha_2)\\
	&&-\frac{2\lambda^2\mu^2(2+\pind)}{(1+\lambda^2)(1+\mu^2)}\,\sigma_M(\alpha_1\wedge\alpha_2).
	\end{eqnarray*}
	By our choice of the local frames, the last term $\mathcal{C}_2$ in \eqref{p5} is given by
	\begin{eqnarray*}
		\mathcal{C}_2&=&2\sum_{k,l=1}^m(f_{t}^*\rn)_{klkl}(1+\sind_{ll})\\
		&=&2(2+\pind)\rn(\df(e_1),\df(e_2),\df(e_1),\df(e_2))=\frac{2\lambda^2\mu^2(2+\pind)}{(1+\lambda^2)(1+\mu^2)}\,\sigma_N.
	\end{eqnarray*}
Thus, $\mathcal{Q}:=\mathcal{C}_1-\mathcal{C}_2$ in \eqref{p5} can be written as
\begin{eqnarray*}
\mathcal{Q}&\!\!\!=\!\!\!&-\frac{2\lambda^2\mu^2(2+\pind)}{(1+\lambda^2)(1+\mu^2)}\,\Bigl(\sigma_M(\alpha_1\wedge\alpha_2)+\sigma_N\Bigr)\\
&&\!\!\!+\frac{4\lambda^2}{(1+\lambda^2)^2}\operatorname{Ric}_M(\alpha_1,\alpha_1)+\frac{4\mu^2}{(1+\mu^2)^2}\operatorname{Ric}_M(\alpha_2,\alpha_2)\\
&\!\!\!=\!\!\!&\frac{2\lambda^2\mu^2(2+\pind)}{(1+\lambda^2)(1+\mu^2)}\,\Bigl(\operatorname{Ric}_M(\alpha_1,\alpha_1)+\operatorname{Ric}_M(\alpha_2,\alpha_2)-\sigma_M(\alpha_1\wedge\alpha_2)-\sigma_N\Bigr)\\
&&\!\!\!+\frac{2\lambda^2\pind}{(1+\lambda^2)(1+\mu^2)}\operatorname{Ric}_M(\alpha_1,\alpha_1)+\frac{2\mu^2\pind}{(1+\lambda^2)(1+\mu^2)}\operatorname{Ric}_M(\alpha_2,\alpha_2)
\end{eqnarray*}
which by definition of the bi-Ricci curvature and by combining with \eqref{p5}  implies the evolution equation \eqref{p8} for $\pind$.
\end{proof}
\begin{lemma}\label{main lemma}
	Let $M$ be a compact Riemannian manifold $M$ of dimension $m>1$ and $N$ be a complete Riemannian surface of bounded geometry.  Suppose that they satisfy the main curvature assumption
	\eqref{curvature main} on the bi-Ricci curvature. Let $[0,T)$ denote the maximal time interval on which the smooth solution of the mean curvature flow $\{F_t\}_{t\in[0,T)}:M\to M\times N$ exists, with the initial condition given by $F_0=\operatorname{Id}_M\times f_0$, and where $f_0:M\to N$ is a strictly area decreasing map. Then the following hold:
	\begin{enumerate}[\normalfont (a)]
		\item The flow remains graphical for all $t\in[0,T)$.
		\smallskip
		\item There exist constants $c_0,c_1>0$ depending on $f_0$ such that
		\begin{equation}
		\pind\ge\frac{2c_0e^{\varepsilon_0t}}{\sqrt{1+c_0^2e^{2\varepsilon_0t}}}\quad\text { and }\quad\vert\df_t\vert^2_{\gm}\le c_1e^{-\varepsilon_0t},\label{est p}
		\end{equation}
		where $f_t:M\to N$ are the smooth maps induced by $F_t$ and where the constant $\epsilon_0$ is defined by
		$$
		\varepsilon_0:=
		\begin{cases}
		\frac{1}{4}\min_M\Ric_M, &\text{if}\quad \min_M\Ric_M\ge 0,\\
		\frac{1}{2}\min_M\Ric_M,  &\text{if}\quad \min_M\Ric_M<0.
		\end{cases}
		$$
	\end{enumerate}
	In particular, if $\,\Ric_M\ge 0$ or if $T<\infty$, then the smooth family $\{f_t\}_{t\in[0,\infty)}$ remains uniformly strictly area decreasing and uniformly bounded in $C^1(M,N)$ for all $t\in[0,T)$. 
\end{lemma}
\begin{proof}
	Since $M$ is compact, the evolving submanifolds will stay graphical at least on some time interval $[0,T_g)$ with $0<T_g\le T$. More precisely, there exist smooth families of diffeomorphisms $\{\varphi_t\}_{t\in[0,T_g)}\subset\operatorname{Diff}(M)$ and maps $\{f_t\}_{t\in[0,T_g)}:M\to N$ such that $F_t\circ\varphi_t^{-1}=\operatorname{Id}_M\times f_t,$ for any $t\in[0,T_g)$. They are given by $\varphi_t=\pi_M\circ F_t$ and $f_t=\pi_N\circ F_t\circ\varphi_t^{-1}$.
	
	The function $\varrho:M\times[0,T_g)\to\real{}$, given by $\varrho(t):=\min\{\pind(x,t):x\in M\}$, is continuous. Since $f_0$ is strictly area decreasing and $M$ is compact, we have $\varrho_0:=\varrho(0)>0$. Let $T_a\le T_g$ be the maximal time such that $\varrho(t)>0$ for all $t\in[0,T_a)$.
	
	The inequality 
	\begin{equation}\label{est2}
	1-\frac{\pind^2}{4}\le\frac{2(\lambda^2+\mu^2)}{(1+\lambda^2)(1+\mu^2)}\le 2\left(1-\frac{\pind^2}{4}\right)
	\end{equation}
	is elementary and, together with the curvature assumption \eqref{curvature main}, it implies that the quantity $\mathcal{Q}$ in equation \eqref{p8a} can be estimated by
	\begin{equation*}\label{hbound}
	\mathcal{Q}\ge \varepsilon_0\pind(4-\pind^2),
	\end{equation*}
	where
	$$
	\varepsilon_0:=
	\begin{cases}
	\frac{1}{4}\min_M\Ric_M, &\text{if}\quad \min_M\Ric_M\ge 0,\\
	\frac{1}{2}\min_M\Ric_M,  &\text{if}\quad \min_M\Ric_M<0.
	\end{cases}
	$$
	
	From the evolution equation \eqref{p8} for $\pind$, we derive the following estimate 
	$$\big(\nabla_{\dt}-\Delta\big)\pind\ge \varepsilon_0\pind(4-\pind^2)
	-\frac{1}{2\pind}\vert\nabla\pind\vert^2,\quad\text { on }M\times[0,T_a).
	$$
	Therefore the parabolic maximum principle shows that on $M\times[0,T_a)$ we get the first estimate in \eqref{est p}, namely
	\begin{equation}\nonumber
	\pind\ge\frac{2c_0e^{\varepsilon_0 t}}{\sqrt{1+c_0^2e^{2\varepsilon_0 t}}},
	\end{equation}
	where $c_0$ is the positive constant determined by $2c_0/\sqrt{1+c_0^2}=\varrho_0$. Therefore $\pind$ cannot become zero in finite time and in particular $T_a=T_g$. Moreover,
	$$\frac{1-\lambda^2}{1+\lambda^2}=\pind-\frac{1-\mu^2}{1+\mu^2}\ge \pind-1\ge \frac{2c_0e^{\varepsilon_0 t}}{\sqrt{1+c_0^2e^{2\varepsilon_0 t}}}-1,$$
	and since $\lambda$ denotes the largest singular value, we get
	$$\vert df_t\vert^2_{\gm}=\lambda^2+\mu^2\le 2\lambda^2\le 2\frac{\sqrt{1+c_0^2e^{2\varepsilon_0t}}-c_0e^{\varepsilon_0 t}}{c_0e^{\varepsilon_0 t}}\le\frac{2}{c_0}e^{-\varepsilon_0t},$$
	from which we obtain the second estimate in \eqref{est p}, now with $c_1:=2/c_0$. It is well-known that the mean curvature flow stays graphical as long as the maps $f_t$ stay bounded in $C^1$. Thus our estimate implies $T_g=T$. 
\end{proof}
\smallskip
\subsection{Estimates for the mean curvature}~

To obtain long-time existence of the flow one needs $C^2$-estimates. To derive such estimates we first prove an estimate on the mean curvature.
\begin{lemma}
	At points where the mean curvature $H$ is non-zero, we have
	\begin{eqnarray}
	\big(\nabla_{\dt}-\Delta\big)\vert H\vert^2\le -2\bigl\vert\nabla\vert H\vert\bigr\vert^2+2\vert A\vert^2\vert H\vert ^2+\mathcal{R},\label{est h}
	\end{eqnarray}
	where $\mathcal{R}$ is the quantity given by
	\begin{eqnarray}
	\mathcal{R}&=&\frac{2\lambda^2\mu^2\vert H\vert^2}{(1+\lambda^2)(1+\mu^2)}
	\big(\BR_M(\alpha_1,\alpha_2)-\sigma_N\big)\nonumber\\
	&&+2\operatorname{Ric}_{M}(v, v)-\frac{2\lambda^2\mu^2\vert H\vert^2}{(1+\lambda^2)(1+\mu^2)} 
	\big(\operatorname{Ric}_M(\alpha_1,\alpha_1)+\operatorname{Ric}_M(\alpha_2,\alpha_2)\big)\nonumber\\
	&&+2\sigma_{N}\vert w\vert^2. \label{mean a}
	\end{eqnarray}
	Here, the vectors $v$ and $w$ are given by
	\begin{equation}\label{vw}
	v:=\frac{\lambda H^{\xi} }{\sqrt{1+\lambda^{2}}} \alpha_{1} + \frac{\mu H^{\eta} }{\sqrt{1+\mu^{2}}} \alpha_{2},\,\, w:=-\frac{\lambda H^{\eta} }{\sqrt{1+\lambda^{2}}} \alpha_{1}+\frac{\mu H^{\xi} }{\sqrt{1+\mu^{2}}} \alpha_{2},
	\end{equation}
	where $\{\alpha_1,\dots,\alpha_m\}$, $\{\xi,\eta\}$ are the special bases arising from the singular value decomposition defined in subsection \ref{frames}.
\end{lemma}
\begin{proof}
	Recall from \cite[Corollary 3.8]{smoczyk1} that
	\begin{eqnarray*}
	\big(\nabla_{\dt}-\Delta\big)|H|^2=-2|\nabla^\perp H|^2
+2|A^H|^2+2\sum_{k=1}^m\rk(\dF(e_k),H,\dF(e_k),H).
	\end{eqnarray*}
	From the Cauchy-Schwarz inequality we have 
	$$|A^H|^2\le|A|^2|H|^2.$$
	Moreover, at points where
	$H\neq 0$,  we have $|\nabla^\perp H|^2\ge |\nabla|H||^2$. Hence, 
	\begin{eqnarray}
	\big(\nabla_{\dt}-\Delta\big)|H|^2&\le& -2|\nabla|H||^2+2|A|^2|H|^2\nonumber\\
	&&+2\sum_{k=1}^m\rk(dF(e_k),H,dF(e_k),H).\label{eq mean}
	\end{eqnarray}
	Let us compute the last curvature term in \eqref{eq mean}, which in the sequel we call
	$$\mathcal{R}:=2\sum_{k=1}^m\rk(\dF(e_k),H,\dF(e_k),H).$$
	We have
	\begin{eqnarray*}
	\mathcal{R}&\!\!\!=\!\!\!&2\sum_{k=1}^m\rk\Big(\tfrac{\alpha_{k} \oplus \lambda_{k} \beta_{k}}{\sqrt{1+\lambda_{k}^{2}}}, H^{\xi} \xi+H^{\eta} \eta, \tfrac{\alpha_{k} \oplus \lambda_{k} \beta_{k}}{\sqrt{1+\lambda_{k}^{2}}}, H^{\xi} \xi + H^{\eta} \eta\Big)\\
	&\!\!\!=\!\!\!&\underbrace{\sum_{k=1}^m\frac{2}{1+\lambda_k^2}\rm\Big(\alpha_k,\tfrac{ \lambda H^{\xi}}{\sqrt{1+\lambda^2}}\alpha_{1}+\tfrac{\mu H^{\eta}  }{\sqrt{1+\mu^2}}\alpha_2,\alpha_{k},\tfrac{ \lambda H^{\xi}}{\sqrt{1+\lambda^2}}\alpha_{1}+\tfrac{\mu H^{\eta}  }{\sqrt{1+\mu^2}}\alpha_2\Big)}_{=:\mathcal{D}_1}\\
	&\!\!\!+\!\!\!&\underbrace{\sum_{k=1}^2\tfrac{2\lambda_k^2}{1+\lambda_k^2} \rn\Big(\beta_{k},  \tfrac{H^{\xi}}{\sqrt{1+\lambda^2}}\beta_{1}+ \tfrac{H^{\eta}}{\sqrt{1+\mu^2}}\beta_{2}, \beta_{k}, \tfrac{H^{\xi}}{\sqrt{1+\lambda^2}}\beta_{1}+ \tfrac{H^{\eta}}{\sqrt{1+\mu^2}}\beta_{2}\Big)}_{=:\mathcal{D}_2}.
	\end{eqnarray*}
	For $\mathcal{D}_2$ we get
	$$\mathcal{D}_2=\frac{2\lambda^{2}\vert H^{\eta}\vert^{2}+2\mu^{2}\vert H^{\xi}\vert^{2}}{\left(1+\lambda^{2}\right)\left(1+\mu^{2}\right)} \sigma_{N}.$$
	In the next step we compute $\mathcal{D}_1$, and use $v$ defined as in \eqref{vw} to obtain
	\begin{eqnarray*}
	\mathcal{D}_1&\!\!\!=\!\!\!&2\,\frac{\mu^{2}\vert H^{\eta}\vert^{2}+\lambda^{2}\vert H^{\xi}\vert^{2}}{\left(1+\lambda^{2}\right)\left(1+\mu^{2}\right)} \sigma_{M}(\alpha_{1}\wedge \alpha_{2})+2|v|^2\sum_{k\geq 3}\rm\Big(\alpha_k,\frac{v}{|v|},\alpha_k,\frac{v}{|v|}\Big)\\
	&\!\!\!=\!\!\!&2\operatorname{Ric}_{M}(v, v)-\frac{2\lambda^{2} \mu^{2}\vert H\vert^{2}}{\left(1+\lambda^{2}\right)\left(1+\mu^{2}\right)}\, \sigma_{M}(\alpha_{1}\wedge \alpha_{2}).
	\end{eqnarray*}
	Thus
	\begin{eqnarray*}
	\mathcal{R}&\!\!\!\!\!=\!\!\!\!\!&2\operatorname{Ric}_{M}(v, v)-\frac{2\lambda^{2} \mu^{2}\vert H\vert^{2}}{\left(1+\lambda^{2}\right)\left(1+\mu^{2}\right)} \sigma_{M}\left(\alpha_{1}, \alpha_{2}\right)+\frac{2\lambda^{2}\vert H^{\eta}\vert^{2}+2\mu^{2}\vert H^{\xi}\vert^{2}}{\left(1+\lambda^{2}\right)\left(1+\mu^{2}\right)} \sigma_{N}\\
	&\!\!\!\!\!=\!\!\!\!\!&\frac{2\lambda^2\mu^2\vert H\vert^2}{(1+\lambda^2)(1+\mu^2)}
	\big({\BR}_M(\alpha_1,\alpha_2)-\sigma_N\big)+2\operatorname{Ric}_{M}(v, v)\nonumber\\
	&&\!\!\!\!\!-\frac{2\lambda^2\mu^2\vert H\vert^2}{(1+\lambda^2)(1+\mu^2)} \big(\operatorname{Ric}_M(\alpha_1,\alpha_1)+\operatorname{Ric}_M(\alpha_2,\alpha_2)\big)\nonumber\\
	&&\!\!\!\!\!+2\sigma_{N}\,\underbrace{\frac{\lambda^{2}\vert H^{\eta}\vert^{2}+\mu^{2}\vert H^{\xi}\vert^{2}+\lambda^2\mu^2\vert H\vert^2}{\left(1+\lambda^{2}\right)\left(1+\mu^{2}\right)}}_{=\vert w\vert^2}.
	\end{eqnarray*}
This proves the lemma.
\end{proof}
\begin{lemma}
	Let us assume the main curvature condition \eqref{curvature main}. At points where the mean curvature $H$ is non-zero, the function $\varTheta:=\pind^{-1}\vert H\vert^2$ satisfies 
	\begin{eqnarray}
	\big(\nabla_{\dt}-\Delta\big)\varTheta\le\pind^{-1}\langle\nabla\varTheta,\nabla\pind\rangle -2\pind^{-1}\left(\operatorname{Ric}_M(w,w)-\sigma_N\vert w\vert^2\right),\label{est theta}
	\end{eqnarray}
	where $w$ is defined as in \eqref{vw}.
\end{lemma}
\begin{proof} 
	From the evolution equation for $\pind$, and from \eqref{curvature main} we get
	\begin{eqnarray}
	\big(\nabla_{\dt}-\Delta\big)\pind&\ge& -\frac{1}{2\pind}|\nabla\pind|^2+2|A|^2\pind\\
	&&+\frac{2\lambda^2\mu^2\pind}{(1+\lambda^2)(1+\mu^2)} \big(\BR_M(\alpha_1,\alpha_2)-\sigma_N\big)\nonumber\\
	&&+\frac{2\pind}{(1+\lambda^2)(1+\mu^2)}\bigl(\lambda^2\operatorname{Ric}_M(\alpha_1,\alpha_1)+\mu^2\operatorname{Ric}_M(\alpha_2,\alpha_2)\bigr).\nonumber\label{est p2}
	\end{eqnarray}
	Then \eqref{est p2}, \eqref{est h}, and the formula
	$$\big(\nabla_{\dt}-\Delta\big)\varTheta-2\pind^{-1}\langle\nabla\pind,\nabla\varTheta\rangle=\pind^{-1}\bigl(\nabla_{\dt}-\Delta\bigr)\vert H\vert^2-\pind^{-2}\vert H\vert^2\bigl(\nabla_{\dt}-\Delta\bigr)\pind$$
	imply, after some cancellations, that at points where $H\neq 0$, it holds
	\begin{eqnarray}\label{theta}
	\big(\nabla_{\dt}-\Delta\big)\varTheta&\!\!\!-\!\!\!&2\pind^{-1}\langle\nabla\pind,\nabla\varTheta\rangle\\
	&\!\!\!\le\!\!\!&-2\pind^{-1}\bigl\vert\nabla\vert H\vert\bigr\vert^2+\frac{1}{2}\pind^{-3}\vert H\vert^2\vert\nabla\pind\vert^2+\mathcal{E}\nonumber,
	\end{eqnarray}
	where
	$$
	\mathcal{E}\!=\!2\varTheta\Big(\frac{\operatorname{Ric}_{M}\left(v,v\right)+\sigma_N\vert w\vert^2}{\vert H\vert^2}-\frac{\lambda^2}{1+\lambda^2}\operatorname{Ric}_M(\alpha_1,\alpha_1)-\frac{\mu^2}{1+\mu^2}\operatorname{Ric}_M(\alpha_2,\alpha_2)\Big).
	$$
	The term $\mathcal{E}$ is of the form $\mathcal{E}=2\varTheta{\mathcal{F}}$ and ${\mathcal{F}}$ might vanish at some points, for example, if $\lambda=\mu=0$ or if $\mu=\vert H^\eta\vert=0$. This shows that we cannot expect the estimate ${\mathcal{F}}<0$ to hold in general. Since we assume $H\neq 0$, the two gradient terms in the first line of \eqref{theta} can be combined and this gives
	\begin{equation}\label{grad}
	-2\pind^{-1}\bigl\vert\nabla\vert H\vert\bigr\vert^2+\frac{1}{2}\pind^{-3}\vert H\vert^2\vert\nabla\pind\vert^2=-\frac{1}{2}\varTheta^{-1}\vert\nabla\varTheta\vert^2-\pind^{-1}\langle\nabla\varTheta,\nabla\pind\rangle.
	\end{equation}
	From the definition of $v,w$ in \eqref{vw}, we  get
	\begin{eqnarray*}
	\operatorname{Ric}_{M}(v,v)&\!\!\!+\!\!\!&\operatorname{Ric}_{M}(w,w)\\
	&\!\!\!=\!\!\!&\Big(\frac{\lambda^2}{1+\lambda^2}\operatorname{Ric}_M(\alpha_1,\alpha_1)+\frac{\mu^2}{1+\mu^2}\operatorname{Ric}_M(\alpha_2,\alpha_2)\Big)\vert H\vert^2.
	\end{eqnarray*}
	Therefore, together with \eqref{grad}, we can simplify \eqref{theta} and finally obtain the desired inequality for $\varTheta$.
\end{proof}
Now observe that
$$
\vert w\vert^2
=\frac{\lambda^2}{1+\lambda^2}\vert H\vert^2+\frac{\mu^2-\lambda^2}{(1+\lambda^2)(1+\mu^2)}\vert H^\xi\vert^2\le\frac{\lambda^2}{1+\lambda^2}\vert H\vert^2\le\vert H\vert ^2.
$$
Let 
\begin{equation}\label{varepsilon1}
\varepsilon_1:={\sup}_N\sigma_N-{\min}_{\{\vert u\vert=1\}}\left(\Ric_M(u,u)\right).
\end{equation}
Then, at points where $H\neq 0$, inequality \eqref{est theta} implies the estimate
\begin{eqnarray}
	 \big(\nabla_{\dt}-\Delta\big)\varTheta&\le&\pind^{-1}\langle\nabla\varTheta,\nabla\pind\rangle+2\max\{0,\varepsilon_1\}\varTheta.\label{est thetab}
\end{eqnarray}
Applying the maximum principle to \eqref{est thetab}, taking into account Lemma \ref{main lemma}, and the fact that $\pind\le 2$, we immediately obtain the following estimate for the mean curvature.
\begin{lemma}\label{lemma mean}
			Let $M$ be a compact Riemannian manifold $M$ of dimension $m>1$ and let $N$ be a complete Riemannian surface of bounded geometry.  Suppose they satisfy the main curvature
			assumption \eqref{curvature main} on the bi-Ricci curvature. Let $[0,T)$ denote the maximal time interval on which the smooth solution of the mean curvature flow $\{F_t\}_{t\in[0,T)}:M\to M\times N$ exists, with the initial condition given by $F_0=\operatorname{Id}_M\times f_0$, and where $f_0:M\to N$ is a strictly area decreasing map. Then the following hold:
		\begin{enumerate}[\normalfont(a)]
			\item The function $\varTheta:=\vert H\vert^2/\pind$ is well-defined for $t\in[0,T)$ and it satisfies
			\begin{equation}\label{est theta2}
			\varTheta\le\max_{t=0}\varTheta\cdot e^{2\max\{0,\varepsilon_1\}\cdot t},\text{ for all }t\in[0,T),
			\end{equation}
			where $\varepsilon_1$ is the constant defined in \eqref{varepsilon1}.
			\medskip
			\item There exists a constant $a_0>0$, depending only on $f_0$, such that
			\begin{equation}\label{est mean2}
			\vert H\vert^2\le a_0\,e^{2\max\{0,\varepsilon_1\}\cdot t},\text{ for all }t\in[0,T).
			\end{equation}
			In particular, if $\,\Ric_M\ge\sup_N\sigma_N$, then $\vert H\vert^2\le a_0$ for all $t\in[0,T)$.
		\end{enumerate}
\end{lemma}

\section{The barrier theorem and an entire graph lemma}\label{barriergraph}
\noindent In the proof of Theorem \ref{main}, we will need the following barrier theorem that generalizes the
well-known barrier theorem for mean curvature flow of hypersurfaces to any codimension. Before we state 
and prove it, we recall the definition of $m$-convexity.
\begin{definition}
	A smooth function $\phi:P\to\real{}$ on a Riemannian manifold $(P,\gp)$ of dimension $p\ge m$ is called $m$-convex at $y\in P$, if the Hessian $D^2\phi$ of $\phi$ at $y$ satisfies
	$$\sum_{k=1}^mD^2\phi(e_k,e_k)\ge 0$$
	for any choice of $\,m$ orthonormal vectors $\{e_1,\dots, e_m\}\in T_yP$.
\end{definition}
\begin{mythm}[Barrier theorem for the mean curvature flow]\label{barrier}~\\
	Let $F_t:M\to (P,\gp)$, $t\in[0,T)$ be a mean curvature flow of a compact manifold $M$ of dimension $m$ into a complete Riemannian manifold $(P,\gp)$ of dimension $p$. Suppose that $\phi:P\to\real{}$ is a smooth function and $c\in\real{}$ a constant such that $\phi$ is $m$-convex on
	$P^c:=\{y\in P:\phi(y)< c\}.$
	If the initial image $F_0(M)$ is contained in $P^c$, then $F_t(M)\subset P^c$ for all $t\in[0,T)$.
\end{mythm}
\begin{proof}
	Define the function $\omega:M\times[0,T)\to\real{}$, given by $\omega=\phi\circ F_t$.
	Since $\dt F_t=H_t$ we get
	$$\dt\omega=D\phi(H_t)$$
	and moreover
	$$\Delta\omega=\operatorname{trace}_{\gind_t}\bigl(F_t^*D^2\phi\bigr)+D\phi(H_t),$$
	where $\Delta$ denotes the Laplace-Beltrami operator on $M$ with respect to the induced metric $\gind_t=F_t^*\gp$. Thus
	$$\dt\omega=\Delta\omega-\operatorname{trace}_{\gind_t}\bigl(F_t^*D^2\phi\bigr).$$
	Since $\phi$ is $m$-convex on $P^c$ and $F_t$ is an immersion, we see that 
	$$\operatorname{trace}_{\gind_t}\bigl(F_t^*D^2\phi\bigr)\ge 0$$
	as long as $F_t(M)\subset P^c$. Since $M$ is compact and $P^c$ is open, we observe that $F_t(M)\subset P^c$ will hold on some maximal time interval $[0,t_0)\subset [0,T)$. It remains to show that $t_0=T$. Assume $t_0<T$. By continuity, we have
	$$\dt\omega\le\Delta\omega$$
	on $[0,t_0]$. Then the strong parabolic maximum principle implies that $\omega<c$ on $[0,t_0]$ which gives $F_{t_0}(M)\subset P^c$. This contradicts the maximality of $t_0$. Thus $t_0=T$ and $F_t(M)\subset P^c$ for all $t\in[0,T)$.
\end{proof}

\begin{remark}
As we pointed out in Remark \ref{remark curvature}, the long-time existence of
the mean curvature flow does not ensure  smooth convergence. However, in some situations, the geometry
of the ambient space forces the submanifolds to stay in a compact region. For instance, if the ambient space
possesses a compact totally convex set $\mathscr{C}$, then we can use Theorem \ref{barrier} with $\phi$
chosen as the 
squared distance function
to $\mathscr{C}$ to show that the flow stays in a compact region. Recently,
Tsai and Wang \cite{tsai} introduced the notion of strongly stable minimal submanifolds.  They proved that
if $\varSigma$ is an $m$-dimensional compact strongly stable minimal submanifold of a Riemannian manifold $P$, then the squared distance function
to $\varSigma$ is $m$-convex in a tubular neighbourhood of $\varSigma$. Moreover,  
if $\varGamma$ is a compact $m$-dimensional submanifold that is $C^1$-close to $\varSigma$,
then the mean curvature flow $\varGamma_t$ with $\varGamma_0=\varGamma$ exists for all time,
and $\varGamma_t$ smoothly converges to
$\varSigma$ as $t\to\infty$.
We refer also to Lotay and Schulze \cite{lotay} for further generalizations and applications of 
the stability result in \cite{tsai}.
\end{remark}

The next lemma turns out to be very useful and it is a direct consequence of the preceding barrier theorem.
\begin{lemma}\label{lemma convergence}
	Let $(M,\gm)$ be a compact and $(N,\gn)$ a complete Riemannian manifold. Suppose $\{f_t\}_{t\in[0,\infty)}$ is uniformly bounded in $C^k(M,N)$, for all $k\ge 1$, and their graphs evolve by mean curvature flow. If there exists a sequence of times $\{t_n\}_{n\in\natural{}}$, with $\lim_{n\to\infty}t_n=\infty$, such that the sequence $\{f_{t_n}\}_{n\in\natural{}}$ converges in $C^0(M,N)$ to a constant map $f_\infty:M\to N$, then the whole flow $\{f_t\}_{t\in[0,\infty)}$ smoothly converges to $f_\infty$.
\end{lemma}
\begin{proof}
	Let $F_t:M\to M\times N$ be the mean curvature flow of $F_0:=\operatorname{Id}_M\times f_0$.
	Then $f_t=\pi_N\circ F_t\circ\varphi_t^{-1}$, where $\varphi_t=\pi_M\circ F_t$, and $\pi_M:M\times N\to M$, $\pi_N:M\times N\to N$ are the projections onto the factors. By assumption, there exist $y\in N$ and a sequence $\{t_n\}_{n\in\natural{}}$, with $\lim_{n\to\infty}t_n=\infty$, such that
	$$\lim_{n\to\infty}\operatorname{dist}_N\bigl(y,f_{t_n}(x)\bigr)=0\,\text { for all }x\in M,$$
	where $\operatorname{dist}_N$ denotes the distance function on $N$. Let $\mathscr{B}(y,r)$ be the geodesic ball of $N$ with radius $r$ centered at the point $y\in N$, and let $\varrho_y:\mathscr{B}(y,r)\to\real{}$ be the function given by
	$\varrho_y(z):=\operatorname{dist}_N(y,z).$
	For sufficiently small $r>0$, $\varrho_y$ is smooth and strictly convex on $\mathscr{B}(y,r)$. Since $M$ is compact, the sets $f_{t_n}(M)$ uniformly tend to $\{y\}$ as $n\to\infty$. Therefore, for any $j\in\natural{}$ there exists a sufficiently large time $t_{n_j}$ such that the image $f_{t_{n_j}}(M)$ is contained in the geodesic ball $\mathscr{B}(y,r/j)$. For a fixed $j$, define the compact set $C_j\subset N$ by 
	$C_j:=\overline{\mathscr{B}(y,r/j)}.$
	Then the function
	$\phi:=\varrho_y\circ\pi_N:M\times N\to\real{}$ given by $\phi(x,z)=\varrho_y(z)$,
	is smooth and convex on its sub-level set
	$$P^{r/j}:=M\times C_j=\{(x,z)\in P:=M\times N:\phi(x,z)\le r/j\}.$$
	Applying  the barrier theorem to $\phi$, we see that $F_t(M)\subset P^{r/j}$ for all $t\ge t_{n_j}$ which is equivalent to $f_t(M)\subset C_j$ for all $t\ge t_{n_j}$. This proves
	$$\lim_{t\to\infty}\operatorname{dist}_N\bigl(y,f_t(x)\bigr)=0\,\text { for all }x\in M,$$
	that is, $\{f_t\}_{t\in [0,\infty)}$ converges uniformly in $C^0(M,N)$ to the constant map $f_\infty:M\to N$, $f_\infty\equiv y$. Thus $\{f_t\}_{t\in[0,\infty)}$ is uniformly bounded in $C^k(M,N)$, for all $k\ge 0$. We claim that this implies 
	$$\lim_{t\to\infty}\Vert f_t\Vert_{C^k(M,N)}=0,\, \text { for all } k\ge 1.$$
	Indeed, if this does not hold, then there exist $k\ge 1$, $\varepsilon>0$ and a sequence $\{t_n\}_{n\in\natural{}}$ with $\lim_{t\to\infty}t_n=\infty$ such that
	$$\Vert f_{t_n}\Vert_{C^k(M,N)}\ge \varepsilon,\, \text { for all } n\in\natural{}.$$
	Since $\{f_{t_n}\}_{n\in\natural{}}$ is uniformly bounded in $C^k(M,N)$, for all $k\ge 0$, the Arzel\`a-Ascoli Theorem implies that there exists a subsequence $\{f_{t_{n_j}}\}_{j\in\natural{}}$ smoothly converging to a limit map 
	$f_{\ast}:M\to N.$
	But the same subsequence already converges in $C^0(M,N)$ to $f_\infty$, so the map $f_{\ast}$ must coincide with $f_\infty$. Thus
	$$\varepsilon\le \lim_{j\to\infty}\Vert f_{t_{n_j}}\Vert_{C^k(M,N)}=\Vert f_\infty\Vert_{C^k(M,N)}=0,$$
	because $Df_\infty=0$ and $k\ge 1$. This contradicts the choices of $k,\varepsilon$ and $\{t_n\}_{n\in\natural{}}$. This completes the proof.
\end{proof}

We will also need the following elementary lemma.
\begin{lemma}[Entire graph lemma]\label{graph}
	Let $f:\Omega\to\real{n}$ be a smooth map on an open domain $\Omega\subset\real{m}$ and  $C^1$-bounded. Then the graph $\varGamma_f$ is complete if and only if $f$ is entire, that is $\,\Omega=\real{m}$.
\end{lemma}
\begin{proof} Consider the two metric spaces $(\Omega, d_{\e})$ and $(\varGamma_f,d_{\gind})$, where $d_{\e}$ denotes the euclidean distance function, and $d_{\gind}$ is the distance function on the graph $\varGamma_f$, induced by its Riemannian metric $\gind$. By the Hopf-Rinow theorem, the metric space $(\varGamma_f,d_{\gind})$ is complete if and only if $(\varGamma_f,\gind)$ is a complete Riemannian manifold. Moreover, since $\Omega$ is an open domain the metric space $(\Omega, d_{\e})$ is complete if and only if $\Omega=\real{m}$. Therefore, it suffices to prove that the metric space $(\Omega, d_{\e})$ is complete if and only if $(\varGamma_f,d_{\gind})$ is complete. The map $F:=\operatorname{Id}_\Omega\times f:\Omega\to\varGamma_f$ provides a homeomorphism between these metric spaces, its inverse is the projection $\pi:\varGamma_f\to\Omega$. Any smooth curve $c:[0,1]\to\Omega$ can be lifted to a smooth curve $\gamma_c:=F\circ c$ on $\varGamma_f$. From $|\gamma'_c(t)|^2=|c'(t)|^2+|df_{c(t)}(c'(t))|^2$ and from the $C^1$-boundedness of $f$ we conclude the existence of a constant $a>0$, independent of $c$, such that
	$$|c'(t)|^2\le|\gamma'_c(t)|^2\le a^2|c'(t)|^2.$$
	Integrating on the interval $[0,1]$, we see that the lengths of $c$ and $F\circ c$
	satisfy
	\begin{equation*}
	L(c)\le L(F\circ c)\le a L(c).\label{estimate1}
	\end{equation*}
	Taking the infimum over all curves connecting 
	$x_1,x_2\in \Omega$, we conclude that
	\begin{equation*}\label{equivmet}
	d_{\e}(x_1,x_2)\le d_{\gind} (F(x_1),F(x_2))\le a\, d_{\e}(x_1,x_2).
	\end{equation*}
	Thus, $F$ and $\pi$ map Cauchy sequences in $(\Omega,d_{\e})$ to Cauchy sequences in $(\varGamma_f,d_{\gind})$ and vice versa. Since $F$ is a homeomorphism, this shows that the completeness of  $(\Omega,d_{\e})$ and $(\varGamma_f,d_{\gind})$ is equivalent.  
\end{proof}
\section{Proofs of the main results}\label{proofs}
\noindent In this section, we will prove Theorems \ref{main} and \ref{minimal}. We need to recall the blow-up analysis of singularities; for example see \cite{chen}.

\begin{proposition}\label{blow}
Let $M$ be a compact $m$-dimensional manifold and let
$F:M\times [0,T)\to P$ be a solution of the mean curvature 
flow \eqref{mcf}, where $(P,\gp)$ is a $p$-dimensional Riemannian manifold with bounded geometry, and $T\le\infty$ its maximal time of existence. Suppose that there exists $x_\infty\in M$, and a sequence 
$\{(x_j,t_j)\}_{j\in\mathbb{N}}$
in $M\times[0,T)$ with $\lim x_j=x_\infty$, $\lim t_j=T$, such that
$$|A(x_j,t_j)|=\max_{(x,t)\in M\times[0, t_j]}|A(x,t)|=:a_j\to\infty.$$
Consider the family of maps
$F_j:M\times[L_j,R_j)
\to(P,a^2_j{\gp})$, $j\in\mathbb{N}$,
given by
$$F_j(x,s):=F_{j,s}(x):=F(x,{s}/{a^{2}_j}+t_j),$$
where
$$L_j:=-a_j^2t_j\quad\text { and }\quad R_j:=\begin{cases}a_j^2(T-t_j)&\text {, if }T<\infty\\
\infty&\text {, if }T=\infty.\end{cases}$$
Then the following hold:
\begin{enumerate}
\item[\textnormal{(a)}]
For each $j\in\natural{}$, the maps $\{F_{j,s}\}_{s\in[L_j,R_j)}$ evolves by mean curvature flow in time $s$. The second fundamental forms $A_{j}$ of $F_{j}$ satisfy
$$
\hspace{+40pt}A_j(x,s)= A(x,s/a_j^2+t_j)\,\,\text{and}\,\,|A_j(x,s)|=a_j^{-1}|A(x,s/a_j^2+t_j)|,
$$
Hence, for any $s\le 0$, $j\in\mathbb{N}$, we have $|A_j|\le 1$ and $|A_j(x_{j},0)|=1.$\footnote{Norms are with respect to the metrics induced by the corresponding immersions.}
\smallskip
\item[\textnormal{(b)}]
For any fixed $s\le 0$, the sequence  $\{(M,F_{j,s}^*(a^2_j{\gp}),x_j)\}_{j\in\mathbb{N}}$ of pointed manifolds smoothly subconverges in the Cheeger-Gromov sense to a connected complete pointed manifold $(M_{\infty},\gind_{\infty}(s),x_{\infty})$, where $M_{\infty}$ is independent of $s$. Moreover, $\{(P,a^2_j{\gp},F_j(x_j,s))\}_{j\in\mathbb{N}}$ smoothly subconverges to the standard euclidean space $(\mathbb{R}^{p},\gind_{\operatorname{euc}},0)$.
\smallskip
\item[\textnormal{(c)}]
There is an ancient solution $F_{\infty}\colon M_{\infty}\times(-\infty,0]\to\mathbb{R}^p$ of \eqref{mcf} such that, for each $s\le 0$, $\{F_{j,s}\}_{j\in\mathbb{N}}$ smoothly subconverges in the Cheeger-Gromov sense to $F_{\infty, s}$. This convergence is uniform with respect to $s$. Additionally, $|A_{F_{\infty}}|\le 1$ and $|A_{F_{\infty}}(x_{\infty},0)|=1.$
\smallskip
\item[\textnormal{(d)}]
If $T=\infty$, then $R_j=\infty$. If $T<\infty$ and the singularity is of type-II, then $R_j\to\infty$.  In both cases, $F_\infty$ can be constructed on  $(-\infty,\infty)$, and gives an eternal solution of \eqref{mcf}.
\end{enumerate}
\end{proposition}

\smallskip
\subsection{Proofs of Theorem \ref{main} and Theorem \ref{minimal}}~

We are now ready to prove our main results starting with Theorem \ref{minimal}.

\smallskip

\noindent\textbf{Proof of Theorem \ref{minimal}.}
Suppose $f:M\to N$ is a smooth and strictly area decreasing minimal map. Since $\pind>0$, $H=0$, and $\dt\pind=0$ we may use the evolution equation \eqref{p8} of $\pind$ in Lemma \ref{evolpind} to conclude
\begin{eqnarray}\label{p9} 
\Delta\pind -\frac{1}{2\pind}\vert\nabla\pind\vert^2+2\pind\vert A\vert^2 +\mathcal{Q}\le 0.
\end{eqnarray}
From the conditions \eqref{curvature main} and \eqref{curvature b}, $\mathcal{Q}$ in \eqref{p9} is non-negative. Hence
$$\Delta\sqrt{\pind}+\sqrt{\pind}\vert A\vert^2=\frac{1}{2\sqrt{\pind}}\Bigl(\Delta\pind-\frac{1}{2\pind}\vert\nabla\pind\vert^2+2\pind\vert A\vert^2\Bigr)\le 0.$$
Integration gives $\vert A\vert^2=0$ and therefore $f$ must be totally geodesic. Once we know that $f$ is totally geodesic, equation \eqref{gradpind} shows that $\nabla\sind=0$ and hence the singular values $\lambda,\mu$ must be constant functions on $M$, in particular $\pind $ is constant. This proves the first part of the theorem.

If $\rank(df)=0$ then clearly $f$ must be constant and $\lambda=\mu=0$. Suppose now that
$\rank(df)>0$. Once we know that $f$ is totally geodesic, equation \eqref{p9} implies $\mathcal{Q}=0$. Therefore, from \eqref{curvature main}, \eqref{curvature b},  $\pind>0 $, and from the definition of $\mathcal{Q}$ we obtain the following equations:
	\begin{eqnarray}
	0&=&\lambda^2\mu^2\big(\BR_M(\alpha_1,\alpha_2)-\sigma_N\big),\label{q1}\\
	0&=&\lambda^2\operatorname{Ric}_M(\alpha_1,\alpha_1),\label{q2}\\
	0&=&\mu^2\operatorname{Ric}_M(\alpha_2,\alpha_2).\label{q3}
	\end{eqnarray}
	
	We claim that $\mathcal{V}:=\operatorname{ker}df$ and $\mathcal{H}:=\bigl(\ker df\bigr)^\perp$ are parallel distributions on $M$, where $\mathcal{H}$ is the horizontal distribution given by the orthogonal complement of $\mathcal{V}$ with respect to the graphical metric $\gind$ on $M$.  The distributions are certainly smooth since at each point $x\in M$, the fiber $\mathcal{V}_x$ is the kernel of the smooth bilinear form $\sind-\gind$ and the nullity of $\sind-\gind$ is fixed, because the eigenvalues of $\sind$ are constant. Since the second fundamental form $A$ vanishes, equation \eqref{sec formula} shows that the Levi-Civita connections of $\gm$ and $\gind$ coincide, that is
\begin{equation}\label{conn}
\nabla_vw=\nabla^{\gm}_vw,\,\text { for all } v,w\in TM.
\end{equation} 
In particular, the geodesics on $M$ with respect to these metrics coincide. Moreover, again by equation \eqref{sec formula}, we get
\begin{equation}\label{par k}
\nabla_vw=\nabla^{\gm}_vw\in\Gamma(\mathcal{V}),\text { for all }v\in TM\text { and }w\in\Gamma(\mathcal{V}).
\end{equation} 
Using the fact that the connections are metric with respect to $\gind$ and $\gm$, and that the two distributions are orthogonal to each other with respect to $\gind$, we see that in addition 
\begin{equation}\label{par kp}
\nabla_vw=\nabla^{\gm}_vw\in\Gamma(\mathcal{H}),\text { for all }v\in TM\text { and }w\in\Gamma(\mathcal{H}).
\end{equation}
Equations \eqref{par k} and \eqref{par kp} imply that the distributions are parallel and involutive. Therefore, by Frobenius' Theorem, for each $x\in M$ there exist unique integral leaves $V_x$ of $\mathcal{V}$ and $H_x$ of $\mathcal{H}$. Since the distributions are parallel and orthogonal to each other,  $V_x$ and $H_x$ are complete and totally geodesic submanifolds of $M$, intersecting orthogonally in $x$. Since $M$ is compact and the integral leaves $V_x$ are the pre-images $K_y$ of points $y\in f(M)$, $V_x$ must be closed and embedded. Thus, $(M,\gm)$ is locally isometric to the Riemannian product of two manifolds $(L,\gind_{L})$ and $(K,\gind_K)$ of non-negative Ricci curvature, and $f:M\to f(M)$ is a submersion. The set $f(M)$ is compact, because $M$ is compact and $f$ continuous. Therefore, if $\rank(df)=1$, then $\upgamma$ must be a closed $1$-dimensional submanifold of $N$, and because $f$ is totally geodesic, this curve must be a geodesic. If $\rank(df)=2$, then $f(M)$ must coincide with $N$, because submersions are open maps and $N$ is connected\footnote{In this article we assume manifolds are connected.}. 

\medskip
\noindent{\bf Claim:} {\em If $\rank(df)=2$, then the horizontal leaves and $(N,\gn)$ are flat and the
surface $N$ is diffeomorphic to a torus $\mathbb{T}^2$ or to a Klein bottle $\mathbb{T}^2/\mathbb{Z}_2$.}
\medskip

\noindent
\textit{Proof of the claim.} Since $(M,\gm)$ is locally a product manifold, the tangent vectors $\alpha_1,\alpha_2$ in the singular value decomposition are given by horizontal vectors. From \eqref{q1}-\eqref{q3} we get
$$\Ric_M(\alpha_1,\alpha_1)=\Ric_M(\alpha_2,\alpha_2)=\sigma_M(\alpha_1\wedge\alpha_2)=0.$$
Since $\mathcal{H}$ is $2$-dimensional and totally geodesic, it is flat. To see that $(N,\gn)$ is flat, we use equation  \eqref{q1} again, and get $\sigma_N\circ f=0$. As we have already seen, $f(M)=N$. Therefore $\sigma_N\equiv 0$. This proves the claim. Since $\gn$ is flat and $N$ is compact, we conclude that $N$ is diffeomorphic to a torus $\mathbb{T}^2$ or to a Klein bottle $\mathbb{T}^2/\mathbb{Z}_2$. \hfill{$\circledast$}

\smallskip\noindent
In particular, this proves that the kernel of the Ricci operator is non-trivial, because $M$ splits locally into the Riemannian product of the fibers and the horizontal leaves.

\smallskip\noindent
If $\rank(df)=1$, and $\alpha$ is a horizontal vector of unit length, then by definition of the singular values $df(\alpha)=\lambda\, \beta$, for a unit tangent vector $\beta$ to the curve $\upgamma$. Thus, if we equip $\upgamma$ with the metric $\lambda^{-2}\gind_{\upgamma}$, then $df$ becomes an isometry. This proves that $f(M,\gm)\to(\gamma,\lambda^{-2}\gind_{\upgamma})$ is a Riemannian submersion. In the same way we see that the map $f:(M,\gm)\to(N,\lambda^{-2}\gn)$ is a Riemannian submersion, if $\lambda=\mu$. 

\smallskip\noindent
It remains to show that the Euler characteristic of $M$ vanishes. Any vector field $W\in\mathfrak{X}(f(M))$ can be lifted in a unique way to a smooth horizontal vector field $\alpha\in\Gamma(\mathcal{H})$ on $M$, that is $df(\alpha)=W\circ f$. In particular, if $W$ is a non-vanishing vector field, then $\alpha$ is non-vanishing since $f$ is a submersion and $\alpha\in\Gamma(\mathcal{H})$. The image $f(M)$ is diffeomorphic to either of $\mathbb{S}^1$, $\mathbb{T}^2$ or  $\mathbb{T}^2/\mathbb{Z}_2$, and there exist non-vanishing vector fields on these target manifolds. Thus, we obtain non-vanishing horizontal vector fields on $M$. By the Poincar\'e-Hopf Theorem this shows that the Euler characteristic $\chi(M)$ vanishes.

This finishes the proof of Theorem \ref{minimal}.\hfill{$\square$}

\smallskip

\noindent\textbf{Proof of Theorem \ref{main}.}
	(a) We already know from Lemma \ref{main lemma} that the flow remains graphical as long as it exists and that all maps $f_t$, $t\in[0,T)$, stay strictly area decreasing. Thus, it remains to show that the maximal time of existence $T$ is infinite. Suppose by contradiction that $T<\infty$. Hence there exists a sequence $\{(x_j,t_j)\}_{j\in\natural{}}$ in $M\times[0,T)$ such that
	$$
	\lim t_j=T,\quad a_j=\max_{(x,t)\in M\times[0,t_j]}|A|(x,t)=|A(x_j,t_j)|\quad\text{and}\quad\lim a_j=\infty.
	$$
	Let $F_j: M\times [-a_j ^2 t_j,0]\to(M\times N,a_j ^2 (g_M\times g_N))$, $j\in\natural{}$,
	be the family of graphs of the maps
	$$f_{s/a_j^2+t_j}:M\to N,\quad s\in[-a_j^2t_j,0].$$
	The singular values of $f_{s/a_j^2+t_j}$, considered as a map between the Riemannian manifolds $(M,a_j ^2\gind_M)$ and $(N,a_j ^2\gind_N)$, coincide with the singular values of the same map, considered as a map between the Riemannian manifolds $(M,\gind_M)$ and $(N,\gind_N)$,	for any $j\in\natural{}$ and any $s\in[-a_j ^2 t_j,0]$.
	Moreover, the mean curvature vector $H_j$ of $F_j$ is related to the mean curvature vector $H$ of $F$ by
	$$
	H_j(x,s)=a_j^{-2}H(x,s/a_j^2+t_j),
	$$
	for any $(x,s)\in M\times[-a_j^2t_j,0].$
	
	Since we assume $T<\infty$, the estimate in \eqref{est mean2} implies that the norm of the mean curvature vector $|H|$ is uniformly bounded in time and since the convergence in Proposition \ref{blow}(c) is smooth, it follows that the ancient solution $F_\infty:M_\infty\to\R^m\times\R^2$ given in Proposition \ref{blow}(c) is a non-totally geodesic complete minimal immersion. From Lemma \ref{main lemma}, it follows that the singular values of $f_t$ remain uniformly bounded as $t\to T$. Then Lemma \ref{graph} implies that $M_\infty=\R^m$.  Hence, $F_\infty:\R^m\to\R^{m+2}$ is an entire minimal strictly area decreasing graph in $\R^{m+2}$ that is uniformly bounded in $C^1$. Due to the Bernstein type result in \cite[Theorem 1.1]{assimos} we obtain that the immersion $F_\infty:\R^m\to\R^{m+2}$ is totally geodesic; see also \cite[Theorem 1.1]{wang}. This contradicts Theorem \ref{blow}(c). Consequently, the maximal time $T$ of existence of the flow must be infinite. This proves Theorem \ref{main}(a).\hfill{$\circledast$}

\smallskip

\noindent(b) Since $\Ric_M\ge 0$, the constant $\varepsilon_0$ in inequality \eqref{est p} is non-negative and therefore $\{f_t\}_{t\in[0,\infty)}$ remains uniformly strictly area decreasing and uniformly bounded in $C^1(M,N)$. This proves part (b) of Theorem \ref{main}.\hfill{$\circledast$}

\smallskip

\noindent (c) The uniform bound on the mean curvature follows directly from \eqref{est mean2}. On the other hand, a uniform $C^2$-bound in the mean curvature flow implies uniform $C^k$-bounds for all $k\ge 2$, if $N$ is complete with bounded geometry. To obtain a uniform $C^2$-bound we need to show that the norm $\vert A\vert$ of the second fundamental form stays uniformly bounded in time. We may then argue in the same way as in part (a) of the proof to derive a contradiction, if $\limsup_{t\to\infty}\vert A\vert=\infty$. We need the uniform $C^1$-bound to apply the entire graph lemma and the Bernstein theorem in \cite{assimos}. This proves part (c).\hfill{$\circledast$}.

\smallskip

\noindent (d) It remains now to prove the last part of Theorem \ref{main}.
\begin{enumerate}[(1)]
	\item This follows from combining (b) and (c).
	\smallskip
	\item We show that for all cases listed in (2) there exists a compact subset $C\subset N$ such that $f_t(M)\subset C$ for all $t$.
	
	\medskip
		\begin{enumerate}[\normalfont(i)]
			\item $\Ric_M>0$.
			From estimate \eqref{est p} in Lemma \ref{main lemma}, it follows that there exist positive constants $c_1,\varepsilon_0$ so that
			$\vert\df_t\vert^2_{\gm}\le c_1e^{-\varepsilon_0t}$, for any $t\ge 0.$
			Clearly $\lim_{t\to\infty}|\df_t|_{\gm}= 0$. Fix a time $t$, take a geodesic $\gamma:[0,1]\to (M,\gm)$ connecting $x,y\in M$, and let  $\varphi:=f_t\circ\gamma$. Thus in terms of the length $L(\gamma)$ of $\gamma$ we get
			\begin{eqnarray*}
				&&\operatorname{dist}_N\big(f_t(x),f_t(y)\big)\le\int_0^1|\varphi'(s)|ds=\int_0^1|(f_t\circ\gamma)'(s)|ds\\
				&&\quad\quad\quad\le\int_0^1|(\df_t)_{\varphi(s)}|_{\gm}|\gamma'(s)|ds=L(\gamma) \int_0^1|(\df_t)_{\varphi(s)}|_{\gm}ds\\
				&&\quad\quad\quad\le L(\gamma)\sqrt{c_1}e^{-\varepsilon_0t/2}.
			\end{eqnarray*}
			Therefore,  $\lim(\operatorname{diam}(f_t(M)))=0$.
			Let $\mathscr{B}(q,r)$ be the geodesic ball of $N$ with radius $r$ centered at a point $q\in N$ and let
			$\varrho_q(y):=\operatorname{dist}_N(q,y),$
			for any $y\in\mathscr{B}(q,r)$.
			Since $N$ has bounded geometry,  there exists a positive constant $r_0<\operatorname{inj}_{\gn}(N)$ depending only on $(N,\gn)$, such that $\varrho_q$ is smooth and strictly convex on $\mathscr{B}(q,r_0)$ for all $q\in N$. Since the diameters of $f_t(M)$ shrink to zero, there exists a sufficiently large time $t_0$ such that 
			$f_{t_0}(M)$ is contained in a geodesic ball
			$\mathscr{B}(p,r_0)$. We may now proceed exactly as in the proof of Lemma \ref{lemma convergence} to show that
			$f_t(M)\subset C:=\overline{\mathscr{B}(p,r_0)},$
			for $t\ge t_0$.
			\smallskip
			\item[(ii)] $N$ is compact. Choose $C:=N$.
			\smallskip
			\item[(iii)] $N$ is diffeomorphic to $\real{2}$. Since the curvature of $N$ is non-positive, the distance function $\varrho_p:N\to\real{}$ to any fixed point $p\in N$ is globally smooth and convex; see \cite[Theorem 4.1]{bishop}. Similarly as in (i), we choose
			$\phi:= \varrho_p\circ\pi_N:M\times N\to \real{}$
			as a globally defined convex function on $P$, and apply Theorem \ref{barrier} to $\phi$ and the set
			$C:=\overline{\mathscr{B}(p,r)},$
			where $r>0$ is chosen so large that $f_0(M)\subset C$. This yields that $f_t(M)\subset C$ for all $t$.
			\smallskip
			\item[(iv)] $N$ is complete and contains a totally convex subset $\mathscr{C}$. In this case, the distance function $\varrho_{\mathscr{C}}(q):=d_N(\mathscr{C},q)$ is globally convex (see \cite[Remarks 4.3(1)]{bishop}). We can proceed as in (iii) with $\phi:=\varrho_{\mathscr{C}}\circ\pi_N$ and the compact set $C\subset N$ chosen as the closure of a sub-level set of $\varrho_{\mathscr{C}}$ that contains $f_0(M)$, yielding $f_t(M)\subset C$ for all $t$.
			\medskip
			\item[(v)] We proceed as in (iv) with $C:=\overline{N^c}$ and $\phi:=\psi\circ\pi_N$.
		\end{enumerate}
	\medskip
This completes the proof of part (2) of (d).
\smallskip
\item
\begin{enumerate}[(i)]
	\item The volume measure $d\mu$ on $\varGamma_f$ evolves by
$\partial_t d\mu=- |H|^2 d\mu.$
	By integration we get
	$$\int_{0} ^{\infty} \left(\int_{M} |H|^2 d\mu\right) dt<\infty.$$
	Hence, there exists a sequence $\{t_n\}_{n\in\mathbb{N}}$, $\lim_{n\to\infty}t_n=\infty$, such that
	\begin{equation}\label{eq minimal}
	\lim_{n\to\infty} \left.\int_{M} |H|^2 d\mu\right\vert_{t=t_n}=0.
	\end{equation}
	Because $\{f_{t_n}\}_{n\in\natural{}}$ is uniformly bounded in $C^k(M,N)$, $k\ge 0$, there exists a subsequence that smoothly converges to a limit map $f_\infty$. By \eqref{eq minimal}, this limit map must be minimal. 
	\smallskip
	\item This follows from Lemma \ref{lemma convergence}.
	\smallskip
	\item This follows from (i) and Corollary \ref{cor b}.
	\smallskip
	\item Assume that $(M,\gm)$, $(N,\gn)$ are real analytic. Since  $\{f_t\}_{t\in[0,\infty)}$ contains a subsequence $\{f_{t_n}\}_{n\in\natural{}}$ that smoothly converges to a totally geodesic map $f_\infty$, a deep result of Leon Simon \cite{simon} shows that the family $\{f_t\}_{t\in[0,\infty)}$ converges smoothly and uniformly to $f_{\infty}$.
\end{enumerate}
\end{enumerate}
This completes the proof of part (3)(iv) and of Theorem \ref{main}.\hfill{$\square$}


\end{document}